\documentclass[11pt]{article}

\usepackage{fullpage}

\usepackage{amsmath}
\usepackage{amssymb}
\usepackage{amsthm}

\usepackage{eucal}

\renewcommand\ge\geqslant
\renewcommand\geq\geqslant
\renewcommand\le\leqslant
\renewcommand\leq\leqslant
\renewcommand\propto\varpropto

\newcommand{\defeq}{:=}

\newtheorem{theorem}{\bf Theorem}[section]
\newtheorem{corollary}{\bf Corollary}[section]

\theoremstyle{definition}
\newtheorem*{example*}{Example}

\newtheorem*{definition*}{Definition}

\theoremstyle{remark}
\newtheorem*{remark*}{Remark}

\newcommand*\xbar[1]{%
  \hbox{%
    \kern0.1em
    \vbox{%
      \hrule height 0.3pt 
      \kern0.35ex
      \hbox{%
        \kern-0.1em
        \ensuremath{#1}%
        \kern-0.1em
      }%
    }%
    \kern0.1em
  }%
} 

\begin{document}

\title{\LARGE\bf Legendre functions of fractional degree: transformations and evaluations}
\author{Robert S. Maier\footnote{Departments of Mathematics and Physics,
University of Arizona, Tucson, Arizona 85721;  
email: rsm@math.arizona.edu; phone: +1 520 621 2617; fax: +1 520 621 8322.}}
\date{}

\maketitle

\abstract{
Associated Legendre functions of fractional degree appear in the solution
of boundary value problems in wedges or in toroidal geometries, and
elsewhere in applied mathematics.  In the classical case when the degree is
half an odd integer, they can be expressed using complete elliptic
integrals.  In this study, many transformations are derived, which reduce
the case when the degree differs from an integer by one-third, one-fourth
or one-sixth to the classical case.  These transformations, or identities,
facilitate the symbolic manipulation and evaluation of Legendre and Ferrers
functions.  They generalize both Ramanujan's transformations of elliptic
integrals and Whipple's formula, which relates Legendre functions of the
first and second kinds.  The proofs employ algebraic coordinate
transformations, specified by algebraic curves.
}



\section{Introduction}
\label{sec:1}
In applied mathematics and theoretical physics, Legendre or associated
Legendre functions occur widely.  Their properties are summarized in many
places~\cite{Olver2010,Erdelyi53}.  The most familiar are ${\rm
  P}_\nu^\mu(z),\allowbreak{\rm Q}_\nu^\mu(z)$, which are defined if
$z\in(-1,1)$, or more generally on the complex $z$-plane with cuts
$(-\infty,-1]$ and~$[1,\infty)$.  Mathematicians call these `Ferrers
    functions' and reserve the term `Legendre functions' for the
    typographically distinct ${P}_\nu^\mu(z),\allowbreak{Q}_\nu^\mu(z)$
    that are defined if $z\in(1,\infty)$, or more generally on the
    $z$-plane with cut~$(-\infty,1]$.  All four functions, said to be of
  degree~$\nu$ and order~$\mu$, satisfy the same second-order ordinary
  differential equation with parameters $\nu$ and~$\mu$, and in the absence
  of cuts may be multi-valued.


The so-called Legendre polynomials ${\rm P}_n^m(z),{\rm Q}_n^m(z)$, which
are really Ferrers functions of integral degree~$n$ and integral order~$m$,
are the most familiar.  For ${\rm P}_n^0(z)$ in~particular, usually written
${\rm P}_n(z)$, no~cuts are required for single-valuedness, and in~fact
${\rm P}_n$ equals~${P}_n$.  Spherical harmonics
$Y_n^m(\theta,\phi)\propto\allowbreak {\rm P}_n^m(\cos\theta){\rm e}^{{\rm
    i}m\phi}$ arise naturally when separating variables in the Laplace and
Helmholtz equations, and are widely used for expansion purposes.
(Of~course `polynomial' is a misnomer; if $m$~is odd, ${\rm P}_n^m(z)$
includes a $\sqrt{1-z^2}$ factor, which appears in~${P}_n^m(z)$
as~$\sqrt{z^2-1}$.)  The use of full ranges for the coordinates, i.e.,
$0\le\theta\le\pi$ and $0\le\phi\le2\pi$, is what causes the quantization
of the degree and order to integers.  It should be noted that Ferrers
functions of half-odd-integer degree and order find application in quantum
mechanics~\cite{Hunter99}.  Like the Legendre polynomials, they are
elementary functions of~$z$.

Ferrers or Legendre functions in which the degree~$\nu$ is a
half-odd-integer and the order is an integer are not elementary but can be
expressed in~terms of the complete elliptic integrals $K=K({\rm m})$,
$E=E({\rm m})$, where ${\rm m}$ (often denoted~$k^2$) is the elliptic
modular parameter.  (For instance, ${\rm P}_{-1/2}(z)$ equals
$(2/\pi)K((1-z)/2)$.)  Ferrers or Legendre functions of \emph{unrestricted}
degree $\nu\notin\mathbb{Z}$ appear in many contexts, such as two Fourier
expansions in the azimuthal coordinate~$\phi$:
\begin{subequations}
    \begin{align}
\left[\cos\theta +{\rm i} \sin\theta\sin\phi\right]^\nu &= 
\sum_{m=-\infty}^\infty \frac{\Gamma(\nu+1)}{\Gamma(\nu+m+1)}\,
{{\rm P}_\nu^m}(\cos\theta)\, {\rm e}^{{\rm i}m\phi},
\label{eq:1a}
\\
\left[\cosh\xi + \sinh\xi\cos\phi\right]^\nu &= 
\sum_{m=-\infty}^\infty \frac{\Gamma(\nu+1)}{\Gamma(\nu+m+1)}\,
{P_\nu^m}(\cosh\xi)\, {\rm e}^{{\rm i}m\phi}.
\label{eq:1b}
    \end{align}
\end{subequations}
In~(\ref{eq:1a}), the left side is a generalization to arbitrary~$\nu$ of a
standard generating function for spherical harmonics \cite[Chap.~VII,
  \S\,7.3]{Courant53}.  Equation~(\ref{eq:1b}) is a `generalized Heine
identity' which has recently attracted attention~\cite{Cohl2011}, and can
be viewed as an analytic continuation of~(\ref{eq:1a}).  In the case
$\nu=-1/2$, it leads to an alternative to the usual multipole expansion of
the $1/\left|{\bf x}-{\bf x}'\right|$ potential~\cite{Cohl2000,Cohl2001}.
Equation~(\ref{eq:1b}) also appears in celestial mechanics, in the analysis
(originally) of planetary perturbations~\cite{Murray99}.  The coefficient
of the mode $\cos(m\phi)$ in the Fourier development of a `disturbing
function' $(1+\alpha^2 - 2\alpha\cos\phi)^{-s}$, where $s>0$ is a
half-odd-integer, is denoted $b_s^{(m)}(\alpha)$ and called a Laplace
coefficient.  By~(\ref{eq:1b}), it can be expressed in~terms of the
Legendre functions~${P}_{-s}^{\pm m}$, and thus in~terms of complete
elliptic integrals.

Legendre (rather than Ferrers) functions with $\nu$ a half-odd-integer and
$\mu$ an integer are commonly called toroidal or `anchor ring' functions,
since harmonics including factors of the form
$P_{n-1/2}^m(\cosh\xi),\allowbreak Q_{n-1/2}^m(\cosh\xi)$ appear when
solving boundary value problems in toroidal domains, upon separating
variables in toroidal coordinates~\cite{Selvaggi2007,Beleggia2009}.  The
efficient calculation of values of toroidal functions, employing
recurrences or other numerical schemes, is well
understood~\cite{Fettis70,Gil2000}.  There is also a literature focusing on
Laplace coefficients, both classical~\cite{Zeipel12} and recent, which
makes contact with hypergeometric expansions.

\smallskip
\emph{Overview of results.}---It is shown that Legendre and Ferrers
functions of any degree~$\nu$ differing from an integer by $\pm1/r$, for
$r=3,4,6$, can be expressed in~terms of like functions of half-odd-integer
degree.  (The order here must be an integer.)  This statement, which leads
to unexpected closed-form expressions in~terms of complete elliptic
integrals, is one consequence of the main results, the large collection of
Legendre identities in \S\,\ref{sec:mainresults}, which facilitate the
rewriting and evaluation of ${P}_{-1/r}^{-\alpha}(\cosh\xi),\allowbreak
{Q}_{-1/r}^{-\alpha}(\cosh\xi)$ and ${\rm
  P}_{-1/r}^{-\alpha}(\cos\theta),\allowbreak {\rm
  Q}_{-1/r}^{-\alpha}(\cos\theta)$, with $\alpha\in\mathbb{C}$ arbitrary.
(The case when $\nu$~differs from $-1/r$ or~$+1/r$ by a non-zero integer is
handled by applying well-known differential recurrences, to shift the
degree.)  The most striking identities may be
\begin{subequations}
\begin{align}
P_{-1/4}^{-\alpha}(\cosh\xi) &= 2^\alpha\sqrt{\mathrm{sech}(\xi/2)}\:{\rm P}_{\alpha-1/2}^{-\alpha}(\mathrm{sech}(\xi/2)),\label{eq:2a}\\
P_{-1/6}^{-\alpha}(\cosh\xi) &= 3^{3\alpha/2}\,\sqrt[4]{\frac{3\sinh(\xi/3)}{\sinh\xi}}\:
{\rm P}_{2\alpha-1/2}^{-\alpha}\left(
\sqrt{\frac{3\sinh(\xi/3)}{\sinh\xi}}\,\cosh(\xi/3)
\right)
,
\label{eq:2b}
\end{align}
\end{subequations}
which hold if $\alpha\in\mathbb{C}$ and $\xi\in(0,\infty)$.  As in all
Legendre identities derived below, the function arguments on the left and
right sides (here trigonometrically parametrized) are algebraically
related.  These two are of special importance because (\ref{eq:1b})~can be
rewritten as
\begin{equation}
  (1+x\cos\phi)^\nu = 
(1-x^2)^{-\nu/2}\!
\sum_{m=-\infty}^\infty \frac{\Gamma(\nu+1)}{\Gamma(\nu+m+1)}\,
{P_\nu^m}(1/\sqrt{1-x^2})\, {\rm e}^{{\rm i}m\phi}
,
\end{equation}
with $x=\tanh\xi$ satisfying $x\in(0,1)$.  This is the little-known Fourier
expansion of $(1+x\cos\phi)^\nu$ (cf.~\cite{Cohl2011,Conway2007}).  It
follows that if $\nu=-1/4$ or $\nu=-1/6$, or more generally if
$\nu$~differs from an integer by $\pm1/r$ with $r=3,4,6$ (as~well as the
classical case $r=2$), the Fourier coefficients of $(1+x\cos\phi)^\nu$ can
be expressed in~terms of complete elliptic integrals.  This too is
unexpected.

Identities such as (\ref{eq:2a}),(\ref{eq:2b}) and the full collection
in~\S\,\ref{sec:mainresults} are closely related to certain function
transformations of Ramanujan.  In his famous notebooks
\cite[Ch.~33]{BerndtV}, he developed a theory of elliptic integrals with
non-classical `signature' $r=3,4,6$, and related them to the classical
integrals, which have signature~$2$.  His theory yields formulas for the
Ferrers functions ${\rm P}_{-1/r}$ in~terms of ${\rm P}_{-1/2}$, or
equivalently the classical complete elliptic integral~$K$, with an
algebraically transformed argument.  (For a compact list of Ramanujan's
transformation formulas that can be written in this way, see
\cite[Lemma~2.1]{Zhou2014}.)  The identities derived here include several
of Ramanujan's rationally parametrized formulas, but such identities as
(\ref{eq:2a}),(\ref{eq:2b}) are more general, in~that they are formulas for
Ferrers or Legendre functions of arbitrary (i.e., non-zero)
order~$-\alpha$.  The parametrization by trigonometric functions is another
novel feature.

\smallskip
\emph{Methods.}---The technique used below for deriving Legendre identities
was developed by considering Whipple's well-known $Q\leftrightarrow P$
transformation formula~\cite[3.3(13)]{Erdelyi53},
\begin{multline}
\label{eq:whipple0}
  Q_\nu^\mu\left((p^2+1)/(p^2-1)\right) = {\rm e}^{\mu\pi{\rm
      i}}\sqrt{\pi/2}\:\Gamma(\nu+\mu+1)\,
  \sqrt{(p^2-1)/2p}\:
  P_{-\mu-1/2}^{-\nu-1/2}\left((p^2+1)/2p\right),
\end{multline}
which holds if $p\in(1,\infty)$.  Whipple's proof of~(\ref{eq:whipple0}),
given in~\cite{Whipple16}, contains the germ of a broadly applicable
method.  It relies on the arguments on left and right,
$L=(p^2+\nobreak1)/\allowbreak(p^2-\nobreak1)$ and $R=(p^2+\nobreak1)/2p$,
being algebraically related by
$(L^2-\nobreak1)\allowbreak(R^2-\nobreak1)=1$.  This relation defines
an algebraic curve, which is parametrized by~$p$, though it could also be
parametrized as $L=z$ and $R=z/\sqrt{z^2-\nobreak1}$, or trigonometrically
as $L=\coth\xi$ and $R=\cosh\xi$, as is common in the literature.

The point is that the correspondence $L\leftrightarrow R$ is an algebraic
change of the independent variable, which leaves Legendre's differential
equation invariant.  What this means is that if $\mathcal{E}_L$~denotes the
second-order differential equation satisfied by the left side as a function
of~$p$, and as~well by the left side with $Q_\nu^\mu$ replaced
by~$P_\nu^\mu$, and if $\mathcal{E}_R$~denotes its counterpart coming from
the right side; then, $\mathcal{E}_L,\mathcal{E}_R$ will be the same.  Once
one has verified this, to prove~(\ref{eq:whipple0}) one needs only to check
that the left and right sides are the \emph{same} element of the
(two-dimensional) solution space of $\mathcal{E}_L=\mathcal{E}_R$.  This
can be confirmed by examining their asymptotic behavior near singular
points.  The many Legendre identities appearing
in~\S\,\ref{sec:mainresults}, relating Legendre and Ferrers functions of
degree $\nu=-1/r$, $r=3,4,6$, to those of other degrees, are all derived in
a similar way, from algebraic curves.

\smallskip
\emph{Applications.}---Legendre functions of fractional degree occur in
many areas of applied mathematics.  One lies in mathematical physics: the
representation theory of certain Lie algebras~\cite{Durand2003b}.  Another
is geometric-analytic: the spectral analysis of Laplacian-like operators on
spaces of negative curvature, which is of interest because of its
connection to quantum chaos~\cite{Gutzwiller85}.  If $\Delta_{\rm LB}$
denotes the Laplace--Beltrami operator on the real hyperbolic
space~$\mathbb{H}^n$, the associated Green's function $(-\Delta_{\rm
  LB}+\nobreak\kappa^2)^{-1}({\bf x}, {\bf x}')$ will be proportional to
$(\sinh d)^{1-n/2}Q_\nu^{n/2-1}(\cosh d)$, where $d$~is the hyperbolic
distance between~${\bf x},{\bf x}'$ and the degree~$\nu$ depends
on~$\kappa^2$.  If~${\kappa^2=0}$, then $\nu=n/2-\nobreak1$; and more
generally, $\nu$~equals $\sqrt{(n-1)^2/4+\kappa^2}-\frac12$.
(See~\cite{Grosche92,Mastumoto2001}, and \cite{Cohl2012} for the
$\kappa^2=0$ case.)  It follows that this Green's function
on~$\mathbb{H}^n$ can be written in~terms of complete elliptic integrals
for an infinite, discrete set of values of the `energy'
parameter~$-\kappa^2<0$.

Another notable application area is the Tricomi problem, which occurs in
two-dimensional transonic potential flow~\cite{Bers58,Rassias90}.  The
Tricomi differential equation on the $\theta$--$\eta$ (i.e.\ hodograph)
plane, $\mathcal{L}u=0$ with $\mathcal{L}=\eta D_\theta^2+\nobreak
D_\eta^2$, has many particular solutions expressible in hypergeometric
functions~\cite[Ch.~XII]{LandauLifschitz6}.  It is not widely appreciated
that the latter are of the special type expressible in terms of Legendre
functions, though it is been observed that many solutions can be obtained
from a fundamental solution (Green's function) of~$\mathcal{L}$ that is
based on~${P}_{-1/6}$ \cite{Lanckau62,Kracht86b}.  In~fact, the so-called
Gellerstedt generalization ${\mathcal{L}}_k=\eta|\eta|^{k-1} D_\theta^2
+\nobreak D_\eta^2$ has a fundamental solution based
on~${P}_{-1/r},{Q}_{-1/r}$, where $r=3,4,6$ for $k=4,2,1$
\cite{BarrosNeto2009}.  Applying such identities as
(\ref{eq:2a}),(\ref{eq:2b}) will express such fundamental solutions
in~terms of complete elliptic integrals.

An additional application area is classical: the separation of variables in
boundary value problems, posed on wedge-shaped domains.  Along this line,
V.~A. Fock~\cite{Fock43} used toroidal coordinates in solving a problem on
a wedge of opening angle~$3\pi/2$, and was led to the fractional-degree
functions $P_{-1/6},Q_{-1/6}$.  More recently, a magnetostatic potential
has been expanded near a cubic corner in modes of the form ${\rm
  P}_\nu^m(\cos\theta){\rm e}^{{\rm i}m\phi} $, where the (irrational)
degree~$\nu$ is corner-specific and known only
numerically~\cite{McKirdy2009}.  If one opening angle of the corner is
increased to~$\pi$, it becomes a right-angled wedge, and the
appropriate~$\nu$ becomes fractional.  In such situations, closed-form
representations like (\ref{eq:2a}),(\ref{eq:2b}) can serve as a check on
numerical work.

In fluid problems on wedges, fractional-degree Ferrers functions ${\rm
  P}_{-1/r},{\rm Q}_{-1/r}$ typically appear in the analysis when the wedge
angle equals $(1-\nobreak \frac1r)\pi$.  This includes problems dealing
with viscous film coating~\cite{Craster94}, solidification~\cite{Hoang98}
and vortex layers~\cite{Crowdy2004}.

\smallskip
\emph{Structure of paper.}---In section~\ref{sec:asymptotics}, the needed
asymptotic behaviors of the Legendre and Ferrers functions are summarized.
Section~\ref{sec:mainresults} contains the main results: the just-mentioned
collection of two dozen identities, or transformation formulas.  In
sections \ref{sec:suppl} and~\ref{sec:derivation}, it is indicated how the
results of \S\,\ref{sec:mainresults} are proved.  The former section is
introductory: it illustrates the method by deriving Whipple's relation and
two similar transformations, one new.  The proof in
section~\ref{sec:derivation} employs distinct algebraic curves when
$r=3,4,6$.  In section~\ref{sec:ellreps}, how to perform integer shifts of
the degree~$\nu$ is explained.  In section~\ref{sec:algebraic}, explicit
formulas are derived, from one of the identities, for ${\rm
  P}_{-1/6}^{-1/4}$, ${P}_{-1/6}^{-1/4}$ and~$Q_{-1/4}^{-1/3}$, each of
which is an elementary (specifically, algebraic) function of its argument.
Finally, section~\ref{sec:curiosity} displays a curiosity: an isolated
identity relating ${\rm P}_{-1/4}^{-1/5}$ to~${\rm P}_{-1/4}^{-1/10}$.

\section{Normalizations and asymptotics}
\label{sec:asymptotics}

The (associated) equation of Legendre is the ordinary differential equation
\begin{equation}
\label{eq:ode}
\frac{{\rm d}}{{\rm d}z}
\left[
(1-z^2) \frac{{\rm d}u}{{\rm d}z}
\right]
 + \left[\nu(\nu+1) - \frac{\mu^2}{1-z^2}\right]u=0
\end{equation}
on the Riemann sphere $\mathbb{C}\cup\{\infty\}$, with degree parameter
$\nu\in\mathbb{C}$ and order parameter $\mu\in\mathbb{C}$.  It is invariant
under $z\mapsto-z$ and under $\mu\mapsto-\mu$ and $\nu\mapsto-1-\nu$.  It
has regular singular points $z=1,-1$ and~$\infty$, with respective pairs of
characteristic exponents $\{\pm\mu/2\}$, $\{\pm\mu/2\}$ and
$\{-\nu,\nu+\nobreak1\}$.  This means that locally, any solution is a
combination of $(z-1)^{\pm \mu/2}$, $(z+1)^{\pm \mu/2}$ and
$(1/z)^{-\nu},(1/z)^{\nu+1}$.  Or rather, this is the generic behavior.  If
the difference between the exponents at any singular point is an integer,
the local behavior of the dominant solution, which comes from the smaller
exponent, will include a logarithmic factor.  This is the source of the
familiar logarithmic behavior when $z\to1$ or~$z\to-1$ of the `polynomials'
${\rm Q}_n^m(z)$ and~${Q}_n^m(z)$ (for integer order~$\mu$, written here
as~$m$).

By convention, the Legendre functions $P_\nu^\mu,Q_\nu^\mu$ are defined so
that when ${\rm Re}\,\mu>0$ and ${{\rm Re}\,\nu>-1/2}$, $P_\nu^{-\mu}$ is
recessive at~$z=1$, and $Q_\nu^\mu$ is recessive at~$z=\infty$.
(See~\cite{Olver97}.)  That is, each is given by a Frobenius series coming
from the larger exponent.  But the question of how best to normalize
$P_\nu^{\mu},Q_\nu^\mu$ (especially the latter) is vexed.  The standard
definition of~$Q_\nu^\mu$ includes factors ${\rm e}^{\mu\pi{\rm i}}$,
$\Gamma(\nu+\mu+1)$ and~$\frac12$, none of which should arguably be
present.  Olver~\cite{Olver97} felt it wise to introduce a new `second
Legendre' function~$\boldsymbol{Q}_{\nu}^\mu$ lacking the first two
factors, so that the standard $Q_\nu^\mu$ equals ${\rm e}^{\mu\pi{\rm
    i}}\Gamma(\nu+\mu+1)\boldsymbol{Q}_{\nu}^\mu$.  In the present paper,
an \emph{ad hoc} function $\widehat Q_\nu^\mu$ that is defined to lack only
the first factor is employed.  That~is, the standard $Q_\nu^\mu$ equals
${\rm e}^{\mu\pi{\rm i}}\widehat Q_{\nu}^\mu$.  Opinion is nearly unanimous that
including the factor ${\rm e}^{\mu\pi{\rm i}}$ in the definition of~$Q_\nu^\mu$
was a mistake, since it may cause $Q_\nu^\mu(z)$ to be non-real even when
$\nu,\mu$ and its argument $z=x>1$ are all real.
(Compare~(\ref{eq:whipple0}).)

The advantage of Olver's $\boldsymbol{Q}_{\nu}^\mu$ is that like
$P_\nu^\mu$ and unlike $Q_\nu^\mu$ (and~$\widehat Q_\nu^\mu$), it is
defined for all $\nu,\mu\in\nobreak\mathbb{C}$; and for any~$z$ not on the
cut $(-\infty,1]$, $\boldsymbol{Q}_{\nu}^\mu(z)$ is an analytic function
  of~$\nu,\mu$.  Also, $\boldsymbol{Q}_{\nu}^\mu$
  and~$\boldsymbol{Q}_{\nu}^{-\mu}$ are identical, much as $P_\nu^\mu$
  and~$P_{-1-\nu}^\mu$ are identical.  Generically, $P_\nu^{-\mu}$
  and~${\boldsymbol{Q}}_\nu^\mu$ are associated respectively with the
  $+\mu/2$ exponent at~$z=1$ and the $\nu+\nobreak1$ exponent
  at~$z=\infty$.  The relevant asymptotics at these `defining' singular
  points are
  \begin{subequations}
  \begin{alignat}{2}
    \label{eq:asympP}
    P_\nu^{-\mu}(z)&\sim\frac1{2^{\mu/2}\Gamma(\mu+1)}\,(z-1)^{\mu/2},&&\qquad z\to1,\\
    {\boldsymbol{Q}}_\nu^{\mu}(z)&\sim\frac{\sqrt{\pi}}{2^{\nu+1}\Gamma(\nu+3/2)}\,(1/z)^{\nu+1},&&\qquad z\to\infty.
  \end{alignat}
  \end{subequations}
These statements are valid if (respectively) $\mu\neq-1,-2,-3,\dots$ and
$\nu\neq-\frac32,-\frac52,-\frac72,\dots$, so that the gamma functions are
finite.  If either gamma function is infinite, the corresponding Legendre
function, when rigorously defined, turns~out to be associated to the other
characteristic exponent: respectively, $-\mu/2$ or~$-\nu$~\cite{Olver97}.
The correct asymptotics in these two degenerate cases are given
in~\cite{Olver2010}.  It should be noted that there are sub-cases of the
degenerate cases in which $P_{\nu}^{-\mu}$,$\boldsymbol{Q}_{\nu}^\mu$ are
\emph{identically zero}.  Specifically, if $M$ equals $1,2,3,\dots$ then
$P_{-M}^M,\dots,P_{M-1}^M\equiv0$; and if $N=1,2,3,\dots$ then
${\boldsymbol{Q}}_{-N-1/2}^{\pm1/2},\dots,{\boldsymbol{Q}}_{-N-1/2}^{\pm(N-1/2)}\equiv0$.
The former fact yields a familiar restriction on the order of spherical
harmonics, but the latter (dual) fact is less well known.

Many formulas and identities involving $\boldsymbol{Q}_{\nu}^\mu$ contain
such obtrusive factors as $\Gamma(\nu+\mu+1)$, and the wish to simplify
these formulas partially justifies the introduction of the less unfamiliar
function~$\widehat Q_\nu^\mu$, and its definition as
$\Gamma(\nu+\nobreak\mu+\nobreak1)\boldsymbol{Q}_{\nu}^\mu$.  For example,
the formula~\cite[\S\,14.1]{Olver97}
\begin{equation}
  u_n^m(z) = (z^2-1)^{m/2}\frac{{\rm d}^m}{{\rm d}z^m} u_n(x)
\end{equation}
(where $\nu,\mu$ are written as $n,m$ because they are taken to be
non-negative integers) holds if $u_n^m$ equals $P_n^m$ or~$\widehat Q_n^m$; but
it does not hold if $u_n^m$ equals~$\boldsymbol{Q}_n^m$.

However, the introduction of the useful function $\widehat Q_\nu^\mu$ comes at
a price.  Owing to the $\Gamma(\nu+\nobreak\mu+\nobreak1)$ factor, $\widehat
Q_\nu^\mu$~is undefined if $\nu+\nobreak\mu$ is a negative integer,
\emph{except} in the abovementioned sub-case: if $N=1,2,3,\dots$, then
$\widehat{Q}_{-N-1/2}^{\pm1/2}(z),\dots,\widehat{Q}_{-N-1/2}^{\pm(N-1/2)}(z)$ are
defined for $z\notin(-\infty,1]$; and~as analytic functions of~$z$, are not
  identically zero.  Informally, this is because in each of these, the
  product of $\Gamma(\nu+\mu+1)$ (infinite) with~$\boldsymbol{Q}_\nu^\mu$
  (zero) is finite and non-zero.  This statement can be made rigorous by a
  limiting argument.  The asymptotic behavior of~$\widehat Q_\nu^\mu$ is
  \begin{equation}
    \label{eq:asympQhat}
    \widehat{Q}_\nu^{\mu}(z)\sim\frac{\sqrt{\pi}\,\Gamma(\nu+\mu+1)}{2^{\nu+1}\Gamma(\nu+3/2)}\,(1/z)^{\nu+1},\qquad z\to\infty,
  \end{equation}
if neither gamma function is infinite.  By continuity, this fact will
suffice for the derivation of Legendre identities that are valid for all
choices of parameter for which $\widehat Q_\nu^\mu$~is defined.

By convention, the Ferrers functions ${\rm P}_\nu^\mu,{\rm Q}_\nu^\mu$ are
related to the Legendre functions ${P}_\nu^\mu,\widehat{Q}_\nu^\mu$ on their
common domains $\pm{\rm Im}\,z>0$, i.e., on the upper and lower
half-planes, by
\begin{subequations}
\label{eq:ferrers}
  \begin{align}
    {\rm P}_\nu^\mu &= {\rm e}^{\pm\mu\pi{\rm i}/2} P_\nu^\mu,\label{eq:ferrersp}\\
    {\rm Q}_\nu^\mu &= {\rm e}^{\mp\mu\pi{\rm i}/2} \widehat{Q}_\nu^\mu \pm {\rm
      i}\frac\pi2\, {\rm e}^{\pm\mu\pi{\rm i}/2} P_\nu^\mu.\label{eq:ferrersq}
  \end{align}
\end{subequations}
Thus on $(-1,1)$, ${\rm P}_\nu^\mu,{\rm Q}_\nu^\mu$ are combinations of
boundary values of the analytic functions ${P}_\nu^\mu,\widehat{Q}_\nu^\mu$.
Equation~(\ref{eq:ferrersq}) is meaningful for all $\nu,\mu$ for which
$\widehat{Q}_\nu^\mu$~is defined, ${\rm Q}_\nu^\mu$~being defined under the
same conditions.  The asymptotic behavior of~${\rm P}_\nu^{-\mu}$ is given
(when $\mu\neq-1,-2,-3,\dots$) by
\begin{equation}
\label{eq:asympPFerrers}
{\rm P}_\nu^{-\mu}(z)\sim\frac1{2^{\mu/2}\Gamma(\mu+1)}\,(1-z)^{\mu/2},\qquad z\to1,  
\end{equation}
and that of ${\rm Q}_\nu^{\mu}(z)$ is discussed in
section~\ref{sec:derivation}.  It should be noted that restricted
to~$(-1,1)$, the functions ${\rm P}_\nu^\mu,2{\rm Q}_\nu^\mu$ are Hilbert
transforms of each other (via `Neumann's integral') when $\mu=0$, and are
related by more complicated integral transforms when
$\mu\neq0$~\cite{Love65}.  This relationship, which is suggested
by~(\ref{eq:ferrers}), is why $Q_\nu^\mu$,~$\boldsymbol{Q}_\nu^\mu$
and~$\widehat Q_\nu^\mu$ should really be defined to be twice as large.
But to maintain compatibility with the past, the factor~$\frac12$ implicit
in their definitions will be kept.

Besides ${\rm P}_\nu^\mu,{\rm Q}_\nu^\mu$ and $P_\nu^\mu,\widehat Q_\nu^\mu$,
the identities in the next section will be expressed for clarity with the
aid of the auxiliary Ferrers function $\xbar{\rm P}_\nu^\mu$ defined by
$\xbar{\rm P}_\nu^\mu(z)={\rm P}_\nu^\mu(-z)$, or
equivalently~\cite[3.4(14)]{Erdelyi53}
\begin{equation}
\label{eq:barP}
\xbar{\rm P}_\nu^\mu = \cos[(\nu+\mu)\pi]\,{\rm P}_\nu^\mu
-\frac2\pi \sin[(\nu+\mu)\pi]\,{\rm Q}_\nu^\mu,
\end{equation}
and the auxiliary Legendre function $\widetilde{P}_\nu^\mu$ defined by
\begin{equation}
\label{eq:widetildeP}
\widetilde{P}_\nu^\mu = \cos[(\nu+\mu)\pi]\,{P}_\nu^\mu
-\frac2\pi \cos(\mu\pi)\,\sin[(\nu+\mu)\pi]\,\widehat{Q}_\nu^\mu.
\end{equation}
The latter \emph{ad hoc} function is less mysterious than it looks:
by~\cite[3.3(10)]{Erdelyi53}, it satisfies
\begin{equation}
\widetilde P_\nu^\mu(x) =\frac12\left[
{\rm e}^{\mu\pi{\rm i}}P_\nu^\mu(-x+{\rm i}0)
+
{\rm e}^{-\mu\pi{\rm i}}P_\nu^\mu(-x-{\rm i}0)
\right], \qquad x>1.
\end{equation}
Both $\xbar{\rm P}_\nu^\mu$ and~$\widetilde{P}_\nu^\mu$ are solutions
of~(\ref{eq:ode}), on the respective (Ferrers or Legendre) domain.

\section{Main results}
\label{sec:mainresults}

Theorems \ref{thm:i4}--\ref{thm:i3p} and their respective corollaries
contain the main results: a collection of two dozen algebraic Legendre
identities, or transformation formulas.  Each Theorem contains a list of
identities in rationally parametrized form, and its corollary gives each
identity in a trigonometrically parametrized form, which may be more useful
in applications.

The identities come from rational curves $\mathcal{C}_r$
and~$\mathcal{C}'_r$, where the `signature'~$r$ equals $3,4$ or~$6$.  The
following numbering scheme is used.  For each~$r$, the identities come in
pairs.  The two pairs coming from~$\mathcal{C}_r$ are labelled
$I_r(i),I_r(\overline i)$ and $I_r(ii),I_r(\overline{ii})$, and the two
pairs coming from~$\mathcal{C}'_r$ are labelled $I'_r(i),I'_r(\overline i)$
and $I'_r(ii),I'_r(\overline{ii})$.  In each of these pairs with one
exception, the Ferrers functions ${\rm P}_{-1/r}^{-\alpha}, \xbar{\rm
  P}_{-1/r}^{-\alpha}$ appear on the left.  The initial pair
$I_r(i),I_r(\overline i)$ is the exception: on the left it has instead the
Legendre functions ${P}_{-1/r}^{-\alpha}, \widetilde{P}_{-1/r}^{-\alpha}$.
The identities (\ref{eq:2a}),(\ref{eq:2b}) in the Introduction appear here
as $I_4(i),I_6(i)$.

The theorems are ordered so that the case $r=4$ is covered first, since the
curves $\mathcal{C}_4,\mathcal{C}'_4$ and their associated identities are
relatively simple; then $r=6$; and finally $r=3$.  The $r=3$ case closely
resembles the $r=6$ case, but is deficient in that the order~$-\alpha$ of
the left-hand function must equal zero.  The $r=3$ identities given in
Theorems \ref{thm:i3},~\ref{thm:i3p} and their corollaries cannot readily be
generalized to non-zero order, and the same is true of the $r=4$ identities
coming from~$\mathcal{C}_4'$.

It should be noted that when the order $-\alpha$~is an integer~$m$, by
applying these identities one can express the Ferrers pair ${\rm
  P}_{-1/r}^m, \xbar{\rm P}_{-1/r}^m$ (or the Legendre pair ${P}_{-1/r}^m,
\widetilde{P}_{-1/r}^m$ appearing on the left of $I_r(i),I_r(\overline i)$)
in~terms of Legendre or Ferrers functions of half-odd-integer degree and
integer order.  Hence, one can express ${\rm P}_{-1/r}^m, {\rm Q}_{-1/r}^m$
(or ${P}_{-1/r}^m, \widehat{Q}_{-1/r}^m$) in~terms of complete elliptic
integrals.  The extension of this result to the case when the degree is
not~$-1/r$, but differs from an integer by~$-1/r$ (or by~$+1/r$) is
explained in section~\ref{sec:ellreps}.

In each identity, the parameter (whether $\xi$, $\theta$ or~$p$) varies
over a specified real interval.  In fact each identity extends by analytic
continuation to the complex domain, to the largest connected open subset
of~$\mathbb{C}$ containing this interval on which both sides are defined.
The only obstruction to their being defined is the requirement that neither
function argument lie on a cut.

\subsection{Signature-4 identities}

\begin{definition*}
  The algebraic $L$--$R$ curve $\mathcal{C}_4$ is the curve
  \begin{equation}
    \label{eq:c4}
    LR^2 + (R^2-2) = 0,
  \end{equation}
  which is rationally parametrized by
  \begin{displaymath}
    L=1-2(1-p^2)=-1+2p^2,\qquad R=1/p,
  \end{displaymath}
  and is invariant under $R\mapsto-R$, which is performed by $p\mapsto -p$.
  An associated prefactor function $A=A(p)$, with limit unity when $p\to1$
  and $(L,R)\to(1,1)$, is
  \begin{displaymath}
    A(p) = \left[2^4\,\frac{(R^2-1)^2}{(L^2-1)^2}\right]^{1/16} \!=\: \sqrt{\frac1p}.
  \end{displaymath}
\end{definition*}

\begin{theorem}
  \label{thm:i4}
  For each pair $u,v$ of Legendre or Ferrers functions listed below, an
  identity
  \begin{displaymath}
    u_{-1/4}^{-\alpha}(L(p)) = 2^\alpha A(p)\,   v_{\alpha-1/2}^{-\alpha}(R(p))
  \end{displaymath}
  of type\/ $I_4$, coming from the curve\/ $\mathcal{C}_4$, holds for the
  specified range of values of the parameter\/ $p$.

\medskip\medskip
\begin{tabular}{llllll}
\hline
Label &$u_{-1/4}^{-\alpha}$ &$v_{\alpha-1/2}^{-\alpha}$ &$p$ range & $L$ range & $R$ range\\
\hline
\hline
(i) & $P_{-1/4}^{-\alpha}$ & ${\rm P}_{\alpha-1/2}^{-\alpha}$ & $(1,\infty)$ & $1<L<\infty$ & $1>R>0$ \\
$(\overline{i})$ & $\csc(\pi/4)\widetilde P_{-1/4}^{-\alpha}$ & $(2/\pi){\rm Q}_{\alpha-1/2}^{-\alpha}$ &  & & \\
\hline
(ii) & ${\rm P}_{-1/4}^{-\alpha}$ & ${P}_{\alpha-1/2}^{-\alpha}$ & $(0,1)$ & $-1<L<1$ & $\infty>R>1$ \\
$(\overline{ii})$ & $\csc(\pi/4)\xbar{\rm P}_{-1/4}^{-\alpha}$ & $(2/\pi)\widehat{Q}_{\alpha-1/2}^{-\alpha}$ &  & & \\
\hline
\end{tabular}
\end{theorem}

\smallskip
To construct trigonometric versions of these identities, one substitutes
$L=\cosh\xi$ and $L=\cos\theta$ into the relation~(\ref{eq:c4}), and solves
for~$R$ as a function of $\xi$ or~$\theta$.  This yields the following.

\begin{corollary}
  \label{cor:i4}
  The following identities coming from\/ $\mathcal{C}_4$ hold for
  $\alpha\in\mathbb{C}$, when $\xi\in(0,\infty)$ and $\theta\in(0,\pi)$.
  \begin{align*}
    &I_4(i):\quad P_{-1/4}^{-\alpha}(\cosh\xi) = 2^\alpha\sqrt{{\rm sech}(\xi/2)}\,{\rm P}_{\alpha-1/2}^{-\alpha}({\rm sech}(\xi/2));&\\
    &I_4(\overline{i}):\quad \text{The same, with
      $P_{-1/4}^{-\alpha},{\rm P}_{\alpha-1/2}^{-\alpha}$ replaced by\/ $\csc(\pi/4)\widetilde
      P_{-1/4}^{-\alpha},(2/\pi){\rm Q}_{\alpha-1/2}^{-\alpha}$};&\\
    &I_4(ii):\quad {\rm P}_{-1/4}^{-\alpha}(\cos\theta) = 2^\alpha\sqrt{\sec(\theta/2)}\,{P}_{\alpha-1/2}^{-\alpha}(\sec(\theta/2));&\\
    &I_4(\overline{ii}):\quad \text{The same, with\/
      ${\rm P}_{-1/4}^{-\alpha},{P}_{\alpha-1/2}^{-\alpha}$ replaced by\/ $\csc(\pi/4)\xbar{\rm P}_{-1/4}^{-\alpha},(2/\pi)\widehat{Q}_{\alpha-1/2}^{-\alpha}$}.&
  \end{align*}
\end{corollary}

\begin{definition*}
  The algebraic $L$--$R$ curve $\mathcal{C}'_4$ is the curve
  \begin{equation}
    \label{eq:c4p}
    (L-1)(R+3)^2 + 2(R-1)^2=0,
  \end{equation}
  which is rationally parametrized by
  \begin{displaymath}
    L=1-2\,\frac{(p-1)^2}{(p+1)^2}=-1+8\,\frac{p}{(p+1)^2},\qquad
    R=1-2\,\frac{p-1}p = -1 + \frac2p,
  \end{displaymath}
  and is invariant under $R\mapsto\allowbreak 4/(R+\nobreak1)-\nobreak1$,
  which is performed by $p\mapsto 1/p$.  An associated prefactor function
  $A=A(p)$, with limit unity when $p\to1$ and $(L,R)\to(1,1)$, is
  \begin{displaymath}
    A(p) = \left[\frac1{2^6}\,\frac{(R-1)^4(R+1)^4}{(L-1)^2(L+1)^4}\right]^{1/16} \!=\: \sqrt{\frac{1+p}{2p}}.
  \end{displaymath}
\end{definition*}

\begin{theorem}
  \label{thm:i4p}
  For each pair $u,v$ of Legendre or Ferrers functions listed below, an
  identity
  \begin{displaymath}
    u_{-1/4}(L(p)) = A(p)\, v_{-1/2}(R(p))
  \end{displaymath}
  of type\/ $I'_4$, coming from the curve\/ $\mathcal{C}'_4$, holds for the
  specified range of values of the parameter\/ $p$.

\medskip\medskip
\begin{tabular}{llllll}
\hline
Label &$u_{-1/4}$ &$v_{-1/2}$ &$p$ range & $L$ range & $R$ range\\
\hline
\hline
(i) & ${\rm P}_{-1/4}$ & ${\rm P}_{-1/2}$ & $(1,\infty)$ & $1>L>-1$ & $1>R>-1$ \\
$(\overline{i})$ & $\frac12\csc(\pi/4)\xbar{\rm P}_{-1/4}$ & $\xbar{\rm P}_{-1/2}$ &  &  &  \\
\hline
(ii) & ${\rm P}_{-1/4}$ & ${P}_{-1/2}$ & $(0,1)$ & $-1<L<1$ & $\infty>R>1$ \\
$(\overline{ii})$ & $\frac12\csc(\pi/4)\xbar{\rm P}_{-1/4}$ & $(2/\pi)\widehat{Q}_{-1/2}$ &  & & \\
\hline
\end{tabular}
\end{theorem}

\smallskip
\begin{remark*}
  Identities $I_4'(i),I_4'(\overline{i})$ were found by Ramanujan; see
  \cite[Chap.~33, Theorems 9.1 and~9.2]{BerndtV} and
  \cite[Lemma~2.1]{Zhou2014}, and compare
  \cite[Proposition~5.7(a)]{Borwein87}.  He used
  $(p-\nobreak1)/(p+\nobreak1)\in(0,1)$ as parameter.
\end{remark*}

To construct trigonometric versions of these identities, one substitutes
$L=\cos\theta$ into the relation~(\ref{eq:c4p}), and solves for~$R$ as a
function of~$\theta$.  This yields the following.

\begin{corollary}
  \label{cor:i4p}
  The following identities coming from\/ $\mathcal{C}'_4$ hold 
  when $\theta\in(0,\pi)$.
  \begin{align*}
    &I'_4(i):\quad {\rm P}_{-1/4}(\cos\theta) = \frac1{\sqrt{1+\sin(\theta/2)}}\,{\rm
      P}_{-1/2}\left(1-4\,\frac{\sin(\theta/2)}{1+\sin(\theta/2)} \right)   ;&\\
    &I'_4(\overline{i}):\quad \text{The same, with\/
      ${\rm P}_{-1/4},{\rm P}_{-1/2}$ replaced by\/ $\tfrac12\csc(\pi/4)\xbar{\rm P}_{-1/4},\xbar{\rm P}_{-1/2}$};&\\
    &I'_4(ii):\quad {\rm P}_{-1/4}(\cos\theta) = \frac1{\sqrt{1-\sin(\theta/2)}}\,{P}_{-1/2}\left(1+4\,\frac{\sin(\theta/2)}{1-\sin(\theta/2)} \right) ;&\\
    &I'_4(\overline{ii}):\quad \text{The same, with\/
      ${\rm P}_{-1/4},{P}_{-1/2}$ replaced by\/ $\tfrac12\csc(\pi/4)\xbar{\rm P}_{-1/4},(2/\pi)\widehat{Q}_{-1/2}$}.&
  \end{align*}
\end{corollary}

\begin{remark*}
  The right-hand arguments equal $-1 + 2\tan^2((\pi-\theta)/4)$ and $-1 +
  2\tan^2((\pi+\theta)/4)$, respectively.
\end{remark*}

\subsection{Signature-6 identities}

\begin{definition*}
  The algebraic $L$--$R$ curve $\mathcal{C}_6$ is the curve
  \begin{equation}
    \label{eq:c6}
    (L^2-1)(4R^2-3)^3 + 27(R^2-1)=0,
  \end{equation}
  which is rationally parametrized by
  \begin{displaymath}
    L=1-54\,\frac{p^2-1}{(p^2-3)^3} = -1 + 2\,\frac{p^4(p^2-9)}{(p^2-3)^3},
    \qquad R= \frac{3+p^2}{4p} = \pm1 \mp \frac{(1\mp p)(p\mp3)}{4p}.
  \end{displaymath}
  It is invariant under $L\mapsto\nobreak -L$ and $R\mapsto-R$, which are
  performed by $p\mapsto3/p$ and $p\mapsto-p$.  An associated prefactor
  function $A=A(p)$, with limit unity when $p\to1$ and $(L,R)\to(1,1)$, is
  \begin{displaymath}
    A(p) = \left[3^6\,\frac{(R^2-1)^2}{(L^2-1)^2}\right]^{1/24} \!=\: \sqrt{\frac{3-p^2}{2p}}.
  \end{displaymath}
\end{definition*}

\begin{theorem}
  \label{thm:i6}
  For each pair $u,v$ of Legendre or Ferrers functions listed below, an
  identity
  \begin{displaymath}
    u_{-1/6}^{-\alpha}(L(p)) = 3^{3\alpha/2}\, A(p)\, v_{2\alpha-1/2}^{-\alpha}(R(p))
  \end{displaymath}
  of type\/ $I_6$, coming from the curve\/ $\mathcal{C}_6$, holds for the
  specified range of values of the parameter\/ $p$.

\medskip\medskip
\begin{tabular}{llllll}
\hline
Label &$u_{-1/6}^{-\alpha}$ &$v_{2\alpha-1/2}^{-\alpha}$ &$p$ range & $L$ range & $R$ range\\
\hline
\hline
(i) & $P_{-1/6}^{-\alpha}$ & ${\rm P}_{2\alpha-1/2}^{-\alpha}$ & $(1,\sqrt{3})$ & $1<L<\infty$ & $1>R>\sqrt3/2$ \\
$(\overline{i})$ & $\csc(\pi/6)\widetilde P_{-1/6}^{-\alpha}$ & $(2/\pi){\rm Q}_{2\alpha-1/2}^{-\alpha}$ &  & & \\
\hline
(ii) & ${\rm P}_{-1/6}^{-\alpha}$ & ${P}_{2\alpha-1/2}^{-\alpha}$ & $(0,1)$ & $-1<L<1$ & $\infty>R>1$ \\
$(\overline{ii})$ & $\csc(\pi/6)\xbar{\rm P}_{-1/6}^{-\alpha}$ & $(2/\pi)\widehat{Q}_{2\alpha-1/2}^{-\alpha}$ &  & & \\
\hline
\end{tabular}
\end{theorem}

\smallskip
To construct trigonometric versions of these identities, one substitutes
$L=\cosh\xi$ and $L=\cos\theta$ into the relation~(\ref{eq:c6}), and solves
for~$R$ as a function of $\xi$ or~$\theta$.  This yields the following.

\begin{corollary}
  \label{cor:i6}
  The following identities coming from\/ $\mathcal{C}_6$ hold for
  $\alpha\in\mathbb{C}$, when $\xi\in(0,\infty)$ and $\theta\in(0,\pi)$.
  \begin{align*}
    &I_6(i):\quad P_{-1/6}^{-\alpha}(\cosh\xi) = 3^{3\alpha/2}\,\sqrt[4]{\frac{3\sinh(\xi/3)}{\sinh\xi}}\:
{\rm P}_{2\alpha-1/2}^{-\alpha}\left(
\sqrt{\frac{3\sinh(\xi/3)}{\sinh\xi}}\,\cosh(\xi/3)
\right);&\\
    &I_6(\overline{i}):\quad \text{The same, with
      $P_{-1/6}^{-\alpha},{\rm P}_{2\alpha-1/2}^{-\alpha}$ replaced by\/ $\csc(\pi/6)\widetilde
      P_{-1/6}^{-\alpha},(2/\pi){\rm Q}_{2\alpha-1/2}^{-\alpha}$};&\\
    &I_6(ii):\quad {\rm P}_{-1/6}^{-\alpha}(\cos\theta) = 3^{3\alpha/2}\,\sqrt[4]{\frac{3\sin(\theta/3)}{\sin\theta}}\:
{P}_{2\alpha-1/2}^{-\alpha}\left(
\sqrt{\frac{3\sin(\theta/3)}{\sin\theta}}\,\cos(\theta/3)
\right);&\\
    &I_6(\overline{ii}):\quad \text{The same, with\/
      ${\rm P}_{-1/6}^{-\alpha},{P}_{2\alpha-1/2}^{-\alpha}$ replaced by\/ $\csc(\pi/6)\xbar{\rm P}_{-1/6}^{-\alpha},(2/\pi)\widehat{Q}_{2\alpha-1/2}^{-\alpha}$}.&
  \end{align*}
\end{corollary}

\begin{remark*}
  The right-hand arguments equal
  $[1+\frac13\tanh^2(\xi/3)]^{-1/2}$ and
  $[1-\frac13\tan^2(\theta/3)]^{-1/2}$, respectively.
\end{remark*}

\begin{definition*}
  The algebraic $L$--$R$ curve $\mathcal{C}_6'$ is the curve
  \begin{equation}
    \label{eq:c6p}
    (L^2-1)(R^2+3)^3 + 27(R^2-1)^2=0,
  \end{equation}
  which is rationally parametrized by
  \begin{displaymath}
    L=1-54\,\frac{(p^2-1)^2}{(p^2+3)^3}= -1 +
    2\,\frac{p^2(p^2-9)^2}{(p^2+3)^3},
    \qquad R= \frac{3-p^2}{2p} = \pm1 \mp \frac{(p\mp1)(3\pm p)}{2p}.
  \end{displaymath}
  It is invariant under $L\mapsto-L$ and $R\mapsto-R$, which are performed
  by $p\mapsto-3/p$ and $p\mapsto-p$, and in~fact under any of the M\"obius
  transformations of~$p$ that permute $p=-3,-1,0,1,3,\infty$, the vertices
  of a regular hexagon on the $p$-sphere.  (Each of these, which form a
  dihedral group of order~$12$, induces a M\"obius transformation of~$L$,
  either $L\mapsto L$ or~$L\mapsto-L$, and one of~$R$.)  An associated
  prefactor function $A=A(p)$, with limit unity when $p\to1$ and
  $(L,R)\to(1,1)$, is
  \begin{displaymath}
    A(p) =
    \left[\frac{3^6}{2^{12}}\,\frac{(R^2-1)^4}{(L^2-1)^2}\right]^{1/24} \!=\:
    \sqrt{\frac{3+p^2}{4p}}.
  \end{displaymath}
\end{definition*}

\begin{theorem}
  \label{thm:i6p}
  For each pair $u,v$ of Legendre or Ferrers functions listed below, an
  identity
  \begin{displaymath}
    u_{-1/6}^{-\alpha}(L(p)) = 3^{3\alpha/2}\,\frac{\Gamma(\alpha+\frac12)}{\sqrt\pi}\, A(p)\, v_{\alpha-1/2}^{-2\alpha}(R(p))
  \end{displaymath}
  of type\/ $I'_6$, coming from the curve\/ $\mathcal{C}'_6$, holds for the
  specified range of values of the parameter\/ $p$.

\medskip\medskip
\begin{tabular}{llllll}
\hline
Label &$u_{-1/6}^{-\alpha}$ &$v_{\alpha-1/2}^{-2\alpha}$ &$p$ range & $L$ range & $R$ range\\
\hline
\hline
(i) & ${\rm P}_{-1/6}^{-\alpha}$ & ${\rm P}_{\alpha-1/2}^{-2\alpha}$ & $(1,3)$ & $1>L>-1$ & $1>R>-1$ \\
$(\overline{i})$ & $\frac12\csc(\pi/6)\xbar{\rm P}_{-1/6}^{-\alpha}$ & $\xbar{\rm P}_{\alpha-1/2}^{-2\alpha}$ &  &  &  \\
\hline
(ii) & ${\rm P}_{-1/6}^{-\alpha}$ & ${P}_{\alpha-1/2}^{-2\alpha}$ & $(0,1)$ & $-1<L<1$ & $\infty>R>1$ \\
$(\overline{ii})$ & $\frac12\csc(\pi/6)\xbar{\rm P}_{-1/6}^{-\alpha}$ & $(2/\pi)\cos(\alpha\pi)\widehat{Q}_{\alpha-1/2}^{-2\alpha}$ &  & & \\
\hline
\end{tabular}
\end{theorem}

\smallskip
\begin{remark*}
The factors $\frac12\csc(\pi/6)$ equal unity and can be omitted; they are
included for consistency with the other theorems in this section.  When
$\alpha=-\frac12,-\frac32,\dots$, the gamma function diverges, but each
$v_{\alpha-1/2}^{-2\alpha}$ is identically zero and the identities are
still valid in a limiting sense.

The $\alpha=0$ case of identities $I_6'(i),I_6'(\overline{i})$ was found by
Ramanujan; see \cite[Chap.~33, Theorem~11.1 and Corollary~11.2]{BerndtV}
and \cite[Lemma~2.1]{Zhou2014}, and compare
\cite[Proposition~5.8]{Borwein87}.  He used $(p-\nobreak1)/2\in(0,1)$ as
parameter.  The generalization to arbitrary~$\alpha$ was given in
hypergeometric notation by Garvan~\cite[(2.32)]{Garvan95}.
\end{remark*}

To construct trigonometric versions of these identities, one substitutes
$L=\cos\theta$ into the relation~(\ref{eq:c6p}), and solves for~$R$ as a
function of~$\theta$.  This yields the following.

\begin{corollary}
  \label{cor:i6p}
  The following identities coming from\/ $\mathcal{C}'_6$ hold for
  $\alpha\in\mathbb{C}$, when $\theta\in(0,\pi)$.
  \begin{align*}
    &I'_6(i):\quad {\rm P}_{-1/6}^{-\alpha}(\cos\theta) = 3^{3\alpha/2}\,\frac{\Gamma(\alpha+\frac12)}{\sqrt\pi}\,\sqrt{\frac{\cos(\pi/6)}{\cos(\pi/6-\theta/3)}}\:
{\rm P}_{\alpha-1/2}^{-2\alpha}\left(
1-2\,\frac{\sin(\theta/3)}{\cos(\pi/6-\theta/3)}
\right);&\\
    &I'_6(\overline{i}):\quad \text{The same, with\/
      ${\rm P}_{-1/6}^{-\alpha},{\rm P}_{\alpha-1/2}^{-2\alpha}$ replaced
  by\/ $\tfrac12\csc(\pi/6)\xbar{\rm P}_{-1/6}^{-\alpha},\xbar{\rm P}_{\alpha-1/2}^{-2\alpha}$};&\\
    &I'_6(ii):\quad {\rm P}_{-1/6}^{-\alpha}(\cos\theta) = 3^{3\alpha/2}\,\frac{\Gamma(\alpha+\frac12)}{\sqrt\pi}\,\sqrt{\frac{\cos(\pi/6)}{\cos(\pi/6+\theta/3)}}\:
{P}_{\alpha-1/2}^{-2\alpha}\left(
1+2\,\frac{\sin(\theta/3)}{\cos(\pi/6+\theta/3)}
\right);&\\
    &I'_6(\overline{ii}):\quad \text{The same, with\/
      ${\rm P}_{-1/6}^{-\alpha},{P}_{\alpha-1/2}^{-2\alpha}$ replaced
  by\/ $\tfrac12\csc(\pi/6)\xbar{\rm P}_{-1/6}^{-\alpha},(2/\pi)\cos(\alpha\pi)\widehat{Q}_{\alpha-1/2}^{-2\alpha}$}.&
  \end{align*}
\end{corollary}

\begin{remark*}
  The right-hand arguments equal $\sqrt3\cot((\pi+\theta)/3)$ and
  $\sqrt3\allowbreak\cot((\pi-\nobreak\theta)/3)$, respectively. The unit
  factors $\frac12\csc(\pi/6)$ are included for consistency with the other
  corollaries in this section.
\end{remark*}

\subsection{Signature-3 identities}

\begin{definition*}
  The algebraic $L$--$R$ curve $\mathcal{C}_3$ is the curve
  \begin{equation}
    \label{eq:c3}
    27(4L-5)^3(R^2-1) - 4 (L-1)(L+1)^3 (4R^2-3)^3 = 0,
  \end{equation}
  which is rationally parametrized by
  \begin{displaymath}
    L=1-54\,\frac{p^2-1}{(p^2-3)^3} = -1 + 2\,\frac{p^4(p^2-9)}{(p^2-3)^3},
    \qquad R= \pm1\mp \frac{(p\mp 1)(3\pm p)^3}{8p^3},
  \end{displaymath}
  and is invariant under $R\mapsto-R$, which is performed by $p\mapsto -p$.
  An associated prefactor function $A=A(p)$, with limit unity when $p\to1$
  and $(L,R)\to(1,1)$, is
  \begin{displaymath}
    A(p) = \left[\frac{3^6}{2^4}\,{\frac{(R-1)^2(R+1)^2}{(L-1)^2(L+1)^6}}\right]^{1/24} \!=\: \sqrt{\frac{(3-p^2)^2}{4p^3}}.
  \end{displaymath}
\end{definition*}

\begin{theorem}
  \label{thm:i3}
  For each pair $u,v$ of Legendre or Ferrers functions listed below, an
  identity
  \begin{displaymath}
    u_{-1/3}(L(p)) = A(p)\, v_{-1/2}(R(p))
  \end{displaymath}
  of type\/ $I_3$, coming from the curve\/ $\mathcal{C}_3$, holds for the
  specified range of values of the parameter\/ $p$.

\medskip\medskip
\begin{tabular}{llllll}
\hline
Label &$u_{-1/3}$ &$v_{-1/2}$ &$p$ range & $L$ range & $R$ range\\
\hline
\hline
(i) & $P_{-1/3}$ & ${\rm P}_{-1/2}$ & $(1,\sqrt{3})$ & $1<L<\infty$ & $1>R>-\sqrt3/2$ \\
$(\overline{i})$ & $\csc(\pi/3)\widetilde P_{-1/3}$ & $(2/\pi){\rm Q}_{-1/2}$ &  & & \\
\hline
(ii) & ${\rm P}_{-1/3}$ & ${P}_{-1/2}$ & $(0,1)$ & $-1<L<1$ & $\infty>R>1$ \\
$(\overline{ii})$ & $\csc(\pi/3)\xbar{\rm P}_{-1/3}$ & $(2/\pi)\widehat{Q}_{-1/2}$ &  & & \\
\hline
\end{tabular}
\end{theorem}

\smallskip
To construct trigonometric versions of these identities, one substitutes
$L=\cosh\xi$ and $L=\cos\theta$ into the relation~(\ref{eq:c3}), and solves
for~$R$ as a function of $\xi$ or~$\theta$.  This yields the following.

\begin{corollary}
  \label{cor:i3}
  The following identities coming from\/ $\mathcal{C}_3$ hold when
  $\xi\in(0,\infty)$ and $\theta\in(0,\pi)$.
  \begin{align*}
    &I_3(i):\quad P_{-1/3}(\cosh\xi) = \sqrt[4]{\frac{3\sinh(\xi/3)\cosh^2(\xi/6)}{\sinh\xi\cosh^2(\xi/2)}}&\\
    &\qquad\qquad\qquad\qquad\qquad\qquad\quad{}\times{\rm P}_{-1/2}\left(\sqrt{\frac{3\sinh(\xi/3)\cosh^2(\xi/6)}{\sinh\xi\cosh^2(\xi/2)}}\:\left[2\cosh(\xi/3)-\cosh(2\xi/3)\right]\right);&\\
    &I_3(\overline{i}):\quad \text{The same, with
      $P_{-1/3},{\rm P}_{-1/2}$ replaced by\/ $\csc(\pi/3)\widetilde
      P_{-1/3},(2/\pi){\rm Q}_{-1/2}$};&\\
    &I_3(ii):\quad {\rm P}_{-1/3}(\cos\theta) = \sqrt[4]{\frac{3\sin(\theta/3)\cos^2(\theta/6)}{\sin\theta\cos^2(\theta/2)}}&\\
    &\qquad\qquad\qquad\qquad\qquad\qquad\quad{}\times{P}_{-1/2}\left(\sqrt{\frac{3\sin(\theta/3)\cos^2(\theta/6)}{\sin\theta\cos^2(\theta/2)}}\:\left[2\cos(\theta/3)-\cos(2\theta/3)\right]\right);&\\
    &I_3(\overline{ii}):\quad \text{The same, with\/
      ${\rm P}_{-1/3},{P}_{-1/2}$ replaced by\/ $\csc(\pi/3)\xbar{\rm P}_{-1/3},(2/\pi)\widehat{Q}_{-1/2}$}.&
  \end{align*}
\end{corollary}

\begin{definition*}
  The algebraic $L$--$R$ curve $\mathcal{C}_3'$ is the curve
  \begin{equation}
    \label{eq:c3p}
    27(4L-5)^3(R^2-1)^2 - 4 (L-1)(L+1)^3 (R^2+3)^3 = 0,
  \end{equation}
  which is rationally parametrized by
  \begin{displaymath}
    L=1-54\,\frac{(p^2-1)^2}{(p^2+3)^3} = -1 +
    2\,\frac{p^2(p^2-9)^2}{(p^2+3)^3},\qquad R= \pm1\mp \frac{(p\mp 1)(3\pm p)^3}{8p^3}.
  \end{displaymath}
  It is invariant under $L\mapsto-L$ and $R\mapsto-R$, which are performed
  by $p\mapsto3/p$ and $p\mapsto-p$.  An associated prefactor function
  $A=A(p)$, with limit unity when $p\to1$ and $(L,R)\to(1,1)$, is
  \begin{displaymath}
    A(p) = \left[\frac{3^6}{2^{16}}\,\frac{(R-1)^4(R+1)^4}{(L-1)^2(L+1)^6}\right]^{1/24} \!=\: \sqrt{\frac{(3+p^2)^2}{16p^3}}.
  \end{displaymath}
\end{definition*}

\begin{theorem}
  \label{thm:i3p}
  For each pair $u,v$ of Legendre or Ferrers functions listed below, an
  identity
  \begin{displaymath}
    u_{-1/3}(L(p)) = A(p)\, v_{-1/2}(R(p))
  \end{displaymath}
  of type\/ $I'_3$, coming from the curve\/ $\mathcal{C}'_3$, holds for the
  specified range of values of the parameter\/ $p$.

\medskip\medskip
\begin{tabular}{llllll}
\hline
Label &$u_{-1/3}$ &$v_{-1/2}$ &$p$ range & $L$ range & $R$ range\\
\hline
\hline
(i) & ${\rm P}_{-1/3}$ & ${\rm P}_{-1/2}$ & $(1,3)$ & $1>L>-1$ & $1>R>-1$ \\
$(\overline{i})$ & $\frac12\csc(\pi/3)\xbar{\rm P}_{-1/3}$ & $\xbar{\rm P}_{-1/2}$ &  &  &  \\
\hline
(ii) & ${\rm P}_{-1/3}$ & ${P}_{-1/2}$ & $(0,1)$ & $-1<L<1$ & $\infty>R>1$ \\
$(\overline{ii})$ & $\frac12\csc(\pi/3)\xbar{\rm P}_{-1/3}$ & $(2/\pi)\widehat{Q}_{-1/2}$ &  & & \\
\hline
\end{tabular}
\end{theorem}

\smallskip
\begin{remark*}
  Identities $I_3'(i),I_3'(\overline{i})$ were found by Ramanujan; see
  \cite[Chap.~33, Theorem~5.6 and Corollary~5.7]{BerndtV} and
  \cite[Lemma~2.1]{Zhou2014}, and compare
  \cite[Proposition~5.7(b)]{Borwein87}.  He used $(p-\nobreak1)/2\in(0,1)$
  as parameter.
\end{remark*}

To construct trigonometric versions of these identities, one substitutes
$L=\cos\theta$ into the relation~(\ref{eq:c3p}), and solves for~$R$ as a
function of~$\theta$.  This yields the following.

\begin{corollary}
  \label{cor:i3p}
  The following identities coming from\/ $\mathcal{C}'_3$ hold when
  $\theta\in(0,\pi)$.
  \begin{align*}
    &I'_3(i):\quad {\rm P}_{-1/3}(\cos\theta) =
    \sqrt{\frac{\cos(\pi/6)}{2\cos(\pi/6-\theta/3) - \cos(\pi/6+2\theta/3)}}&\\
&\qquad\qquad\qquad\qquad\qquad\qquad\quad{}\times{\rm P}_{-1/2}\left(1-2\:\frac{2\sin(\theta/3) + \sin(2\theta/3)}{2\cos(\pi/6-\theta/3)- \cos(\pi/6+2\theta/3)}
\right);&\\
    &I'_3(\overline{i}):\quad \text{The same, with\/
      ${\rm P}_{-1/3},{\rm P}_{-1/2}$ replaced
  by $\tfrac12\csc(\pi/3)\xbar{\rm P}_{-1/3},\xbar{\rm P}_{-1/2}$};&\\
    &I'_3(ii):\quad {\rm P}_{-1/3}(\cos\theta) =
\sqrt{\frac{\cos(\pi/6)}{2\cos(\pi/6+\theta/3) - \cos(\pi/6-2\theta/3)}}&\\
&\qquad\qquad\qquad\qquad\qquad\qquad\quad{}\times{P}_{-1/2}\left(1+2\:\frac{2\sin(\theta/3) + \sin(2\theta/3)}{2\cos(\pi/6+\theta/3)- \cos(\pi/6-2\theta/3)}
\right);&\\
    &I'_3(\overline{ii}):\quad \text{The same, with\/
      ${\rm P}_{-1/3},{P}_{-1/2}$ replaced
  by $\tfrac12\csc(\pi/3)\xbar{\rm P}_{-1/3},(2/\pi)\widehat{Q}_{-1/2}$}.&
  \end{align*}
\end{corollary}

\section{Additional results}
\label{sec:suppl}

This section presents three additional transformation theorems for Legendre
functions, based on rational or algebraic transformations of the
independent variable, and introduces the fundamental proof technique.
Theorem~4.1 relates the four functions ${\rm P}_{\alpha-1/2}^{2\alpha},
\xbar{\rm P}_{\alpha-1/2}^{2\alpha},\allowbreak {P}_{\alpha-1/2}^{2\alpha},
\widehat{Q}_{\alpha-1/2}^{2\alpha}$, and Theorem~4.2 (equivalent to
Whipple's transformation formula) relates
${P}_{\alpha-1/2}^{-\beta},\widehat{Q}_{\beta-1/2}^{-\alpha}$.
Theorem~4.1, or an equivalent, has appeared in the setting of `generalized'
(associated) Legendre functions; compare \cite{Meulenbeld60} and
\cite[\S\,4]{Virchenko2001}.  It is not well known.  Whipple's formula is
more familiar, in~part because it has two free parameters rather than one,
but the proof indicated below is new.  Theorem~4.3 relates the functions
${P}_{2\alpha-1/2}^{\alpha},\widehat{Q}_{2\alpha-1/2}^{\alpha}$, and is
unexpected.

The calculations in the proofs employ the calculus of Riemann P-symbols,
which is classical~\cite{Poole36}.  For any homogeneous second-order
ordinary differential equation $\mathcal{L}u=0$ on the Riemann sphere
$\mathbb{P}^1\defeq\mathbb{C}\cup\{\infty\}$ which is Fuchsian, i.e., has
only regular singular points, the P-symbol tabulates the singular points
and the two characteristic exponents associated to each point.  For
example, Legendre's equation~(\ref{eq:ode}) has P-symbol
\begin{equation}
  \left\{\!
  \begin{array}{ccc|c}
    1 & -1 & \infty & z \\
    \hline
    -\mu/2 & -\mu/2 & -\nu & \\
    \mu/2 & \mu/2 & \nu+1
  \end{array}
  \!\right\},
\end{equation}
in which the order of the points and that of the exponents are not
significant.  As mentioned, ${\rm P}^{-\mu}_\nu,\xbar{\rm
  P}^{\,-\mu}_\nu,\allowbreak P^{-\mu}_\nu$ and $\widehat Q^{\mu}_\nu$ are
Frobenius solutions associated respectively with the exponent $\mu/2$
at~$z=1$, the exponent $\mu/2$ at~$z=-1$, the exponent $\mu/2$ at~$z=1$ and
the exponent $\nu+1$ at~$z=\infty$.

Changes of variable applied to an equation $\mathcal{L}u=0$ affect its
P-symbol in predictable ways.  For instance, if $w(z)=(z-\nobreak z_0)^c
u(z)$ is a linear change of the dependent variable, the transformed
equation $\widetilde {\mathcal{L}} w=0$, i.e.,
$\mathcal{L}\left[(z-z_0)^{-c}w\right]=0$, will have its exponents at~$z=z_0$
shifted upward by~$c$ relative to those of $\mathcal{L}u=0$, and those at
$z=\infty$ similarly shifted downward.  In interpreting this statement one
should note that any ordinary, i.e.\ non-singular point has
exponents~$0,1$.

Any rational map $f\colon {\mathbb{P}}^1_{\tilde z} \to
{\mathbb{P}}^1_{z}$, i.e., rational change of the independent variable of
the form $z=f(\tilde z)$, will lift a Fuchsian differential equation
$\mathcal{L}u=0$ on~$\mathbb{P}^1_z$ to a Fuchsian equation $\widetilde
{\mathcal{L}}\tilde u=0$ on~$\mathbb{P}^1_{\tilde z}$, the P-symbol of
which can be calculated.  The simplest case is when $f$~is a homography
(also called a linear fractional or M\"obius transformation), so that
$z=f(\tilde z)=(A\tilde z+\nobreak B)/\allowbreak (C\tilde z+\nobreak D)$
with $AD-\nobreak BC\neq0$.  In this case $f$~provides a one-to-one
correspondence between the points of the $z$-sphere and those of the
$\tilde z$-sphere, and exponents are unaffected by lifting: $\tilde
z_0\in\mathbb{P}^1_{\tilde z}$ is a singular point of
$\widetilde{\mathcal{L}}\tilde u=0$ if and only if $f(\tilde
z_0)\in\mathbb{P}_z^1$ is one of $\mathcal{L}u=0$; and the exponents are
the same.

More generally, if $f$~is a rational function with $z-z_0\sim{\rm
  const}\times\allowbreak (\tilde z-\nobreak\tilde z_0)^k$, so that $f^{-1}(z_0)$
equals~$\tilde z_0$ with multiplicity~$k$, the exponents at the lifted
point $\tilde z=\tilde z_0$ will be $k$~times those at~$z=z_0$.  (This
assumes that the points $z_0,\tilde z_0$ are finite; if $z_0,\tilde z_0$
are~$\infty$ then $z-\nobreak z_0,\allowbreak \tilde z-\nobreak \tilde z_0$
must be replaced by $1/z,1/\tilde z$.)  Thus by an appropriate lifting, and
if necessary an additional shifting of exponents (performed by an
appropriate linear change of the dependent variable), it may be possible to
convert the exponents of a singular point to~$0,1$.  That is, under
rational lifting a singular point of a differential equation may
`disappear': become an ordinary point, in a neighborhood of which each
Frobenius solution is analytic.

The calculus of P-symbols is a powerful tool for exploring the effect of
changes of variable on Fuchsian differential equations, but a P-symbol does
not, in general, uniquely determine such an equation, or even its solution
space.  If a second-order equation has $m$~specified singular points on the
Riemann sphere ($m\ge3$), it~and its two-dimensional solution space are
determined by the $2m$~exponents and by $m-\nobreak3$ \emph{accessory
  parameters}\cite{Poole36}.  To prove equality between two second-order
differential equations with more than three singular points, which have the
same singular point locations and characteristic exponents but which have
been obtained by different liftings, one must work~out the lifted equations
explicitly, and compare them term-by-term.

\subsection{Homographic identities}

Theorem~\ref{thm:m} below is based on an algebraic change of variable, from
$L$ to~$R$, which is relatively simple: it is a homography of the Riemann
sphere.  The associated curve is denoted~$\mathcal{M}$, after M\"obius, and
the resulting identities are said to be of type~$M$.

\begin{definition*}
  The algebraic $L$--$R$ curve $\mathcal{M}$ is the curve
  \begin{equation}
    \label{eq:m}
    (L+1)(R+1)-4=0,
  \end{equation}
  which is rationally parametrized by
  \begin{displaymath}
    L=-1+2p, \qquad R=-1 + 2/p,
  \end{displaymath}
  and is invariant under $L\leftrightarrow R$, which is performed by
  $p\mapsto 1/p$.  An associated prefactor function $A=A(p)$, equal to
  unity when $p=1$ and $(L,R)=(1,1)$, is
  \begin{displaymath}
    A(p) = \left[\frac{2}{L+1}\right]^{1/2} = \left[\frac{R+1}{2}\right]^{1/2} = \sqrt{\frac1p}.
  \end{displaymath}
\end{definition*}

\begin{theorem}
  \label{thm:m}
  For each pair $u,v$ of Legendre or Ferrers functions listed below, an
  identity
  \begin{displaymath}
    u_{\alpha-1/2}^{-2\alpha}(L(p)) = A(p)\, v_{\alpha-1/2}^{-2\alpha}(R(p))
  \end{displaymath}
  of type\/ $M$, coming from the curve\/ $\mathcal{M}$, holds for the
  specified range of values of the parameter\/ $p$.

\medskip\medskip
\begin{tabular}{llllll}
\hline
Label &$u_{\alpha-1/2}^{-2\alpha}$ &$v_{\alpha-1/2}^{-2\alpha}$ &$p$ range & $L$ range & $R$ range\\
\hline
\hline
(i) & $P_{\alpha-1/2}^{-2\alpha}$ & ${\rm P}_{\alpha-1/2}^{-2\alpha}$ & $(1,\infty)$ & $1<L<\infty$ & $1>R>-1$ \\
$(\overline{i})$ & $(2/\pi)\cos(\alpha\pi)\widehat{Q}_{\alpha-1/2}^{-2\alpha}$ & $\xbar{\rm P}_{\alpha-1/2}^{-2\alpha}$ &  & & \\
\hline
(ii) & ${\rm P}_{\alpha-1/2}^{-2\alpha}$ & ${P}_{\alpha-1/2}^{-2\alpha}$ & $(0,1)$ & $-1<L<1$ & $\infty>R>1$ \\
$(\overline{ii})$ & $\xbar{\rm P}_{\alpha-1/2}^{-2\alpha}$ & $(2/\pi)\cos(\alpha\pi)\widehat{Q}_{\alpha-1/2}^{-2\alpha}$ &  & & \\
\hline
\end{tabular}
\end{theorem}

\smallskip
To construct trigonometric versions of these identities, one substitutes
$L=\cosh\xi$ and $L=\cos\theta$ into the relation~(\ref{eq:m}), and solves
for~$R$ as a function of $\xi$ or~$\theta$.  This yields the following.

\begin{corollary}
  The following identities coming from\/ $\mathcal{M}$ hold for
  $\alpha\in\mathbb{C}$, when $\xi\in(0,\infty)$ and $\theta\in(0,\pi)$.
  \begin{align*}
    &M(i):\quad P_{\alpha-1/2}^{-2\alpha}(\cosh\xi) = {\mathrm{sech}}(\xi/2)\:
{\rm P}_{\alpha-1/2}^{-2\alpha}\left(
1-2\tanh^2(\xi/2)
\right);&\\
    &M(\overline{i}):\quad \text{The same, with
      $P_{\alpha-1/2}^{-2\alpha},{\rm P}_{\alpha-1/2}^{-2\alpha}$ replaced by\/ 
  $(2/\pi)\cos(\alpha\pi)\widehat{Q}_{\alpha-1/2}^{-2\alpha},\xbar{\rm P}_{\alpha-1/2}^{-2\alpha}$};&\\
    &M(ii):\quad {\rm P}_{\alpha-1/2}^{-2\alpha}(\cos\theta) = \sec(\theta/2)\:
{P}_{\alpha-1/2}^{-2\alpha}\left(
1+2\tan^2(\theta/2)
\right);&\\
    &M(\overline{ii}):\quad \text{The same, with\/
      ${\rm P}_{\alpha-1/2}^{-2\alpha},{P}_{\alpha-1/2}^{-2\alpha}$ replaced by\/ $\xbar{\rm P}_{\alpha-1/2}^{-2\alpha},(2/\pi)\cos(\alpha\pi)\widehat{Q}_{\alpha-1/2}^{-2\alpha}$}.&
  \end{align*}
\end{corollary}

\begin{remark*}
  Owing to the invariance under $L\leftrightarrow R$, the pairs
  $M(i),M(\overline{i})$ and $M(ii),M(\overline{ii})$ are equivalent, up to
  parametrization.  When $\alpha$~is a half-odd-integer, these identities
  degenerate or become singular.  If $\alpha=-\frac12,\-\frac32,\dots$,
  both sides of each identity equal zero.  If
  $\alpha=\frac12,\frac32,\dots$, the factor $\cos(\alpha\pi)$ in
  $M(\overline{i})$ and~$M(\overline{ii})$, which equals zero, is
  compensated~for by the factor~$\widehat Q_{\alpha-1/2}^{-2\alpha}$, which
  diverges.  These two identities are still valid in a limiting sense, and
  the singular behavior can be removed by rewriting them in terms of
  Olver's function $\boldsymbol{Q}_{\alpha-1/2}^{-2\alpha}=\allowbreak
  \widehat Q_{\alpha-1/2}^{-2\alpha}/\Gamma(-\alpha+1/2)$, using the
  reflection formula $\cos(\alpha\pi) =\allowbreak
  \pi/\allowbreak\Gamma(\alpha+\nobreak1/2)\allowbreak\Gamma(-\alpha+1/2)$.
\end{remark*}

\begin{proof}[Proof of Theorem~\ref{thm:m}]
Functions $u,v$ satisfy Legendre's equation~(\ref{eq:ode}), of degree
$\nu=\alpha-\nobreak1/2$ and order $\mu=-2\alpha$, if and only if $u(L(p))$
and $A(p)v(R(p))$ both satisfy a certain second-order differential equation
with independent variable~$p$, which is obtained by lifting.  This can be
checked by a P-symbol calculation, noting that the inverse images of the
singular points $1,-1,\infty$ are $p=1,0,\infty$ under $L=L(p)$, and
$p=1,\infty,0$ under $R=R(p)$.  The left and right P-symbols are
\begin{equation}
  \left\{\!
  \begin{array}{ccc|c}
    1 & -1 & \infty & L(p) \\
    \hline
    \alpha & \alpha & -\alpha+1/2 & \\
    -\alpha & -\alpha & \alpha+1/2
  \end{array}
  \!\right\}
  =
  \left\{\!
  \begin{array}{ccc|c}
    1 & 0 & \infty & p \\
    \hline
    \alpha & \alpha & -\alpha+1/2 & \\
    -\alpha & -\alpha & \alpha+1/2
  \end{array}
  \!\right\},
\end{equation}
\begin{align}
  \sqrt{\frac1p}
  \left\{\!
  \begin{array}{ccc|c}
    1 & -1 & \infty & R(p) \\
    \hline
    \alpha & \alpha & -\alpha+1/2 & \\
    -\alpha & -\alpha & \alpha+1/2
  \end{array}
  \!\right\}
  &=
  \sqrt{\frac1p}
  \left\{\!
  \begin{array}{ccc|c}
    1 & \infty & 0 & p \\
    \hline
    \alpha & \alpha & -\alpha+1/2 & \\
    -\alpha & -\alpha & \alpha+1/2
  \end{array}
  \!\right\}\nonumber\\
  &=
  \left\{\!
  \begin{array}{ccc|c}
    1 & \infty & 0 & p \\
    \hline
    \alpha & \alpha+1/2 & -\alpha & \\
    -\alpha & -\alpha+1/2 & \alpha
  \end{array}
  \!\right\},
\end{align}
which are equivalent.  Put differently, the $L\mapsto R$ homography,
followed by a linear change of the dependent variable coming from the
prefactor $A(p)=\sqrt{1/p}$, takes Legendre's equation to itself.

It remains to show that for each $u,v$ listed in the theorem, the functions
$u(L(p))$ and $A(p)v(R(p))$ are the \emph{same} element of the
two-dimensional solution space of the lifted equation.  First, notice that
this is true up~to some constant factor, since each of $u,v$ is a Frobenius
solution associated with a singular point and one of its exponents, and the
points and exponents correspond.  In identity $M(\overline{i})$ for
example, $u=u(L) =\allowbreak (2/\pi)\cos(\alpha\pi)\widehat
Q_{\alpha-1/2}^{-2\alpha}(L)$ is a Frobenius solution at~$L=\infty$ with
exponent $\alpha+1/2$, and $v=v(R) = \xbar{\rm
  P}_{\alpha-1/2}^{-2\alpha}(R)$ is one at~$R=-1$ with exponent~$\alpha$.
Examining the above P-symbols reveals that for~$M(\overline{i})$, both
$u(L(p))$ and $A(p)v(R(p))$, which are Frobenius solutions of the lifted
equation, are associated to its singular point $p=\infty$ and the
exponent~$\alpha+1/2$; so they must be constant multiples of each other.

Finally, one must check that for each identity, the constant of
proportionality equals unity.  This follows by comparing the left and right
sides near the common singular point.  Identity $M(\overline{i})$ is
typical.  To compare left and right asymptotics, one uses
\begin{subequations}
\begin{alignat}{2}
\widehat Q_{\alpha-1/2}^{-2\alpha}(L) &\sim \frac{\sqrt{\pi}\,
  \Gamma(-\alpha+1/2)}{2^{\alpha+1/2}\,
  \Gamma(\alpha+1)}\:(1/L)^{\alpha+1/2},&& \qquad L\to\infty,\\
\xbar{\rm P}_{\alpha-1/2}^{-2\alpha}(R) &\sim
\frac{\sqrt{\pi}}{2^\alpha\,\Gamma(2\alpha+1)}\:(1+R)^\alpha,&&  \qquad R\to-1,
\end{alignat}
\end{subequations}
which come from (\ref{eq:asympQhat}),(\ref{eq:asympP}).  A bit of
calculation, using the duplication formula for the gamma function, shows
that the two sides behave identically at the $p=\infty$ singular point;
specifically,
\begin{equation}
  u(L(p)),\,A(p)v(R(p))\sim
  \frac{\sqrt{\pi}}{\Gamma(2\alpha+1)}\,(1/p)^{\alpha+1/2}, \qquad p\to\infty.
\end{equation}
The other three identities are proved similarly.
\end{proof}

\subsection{Whipple's formula}

Theorem~\ref{thm:w2} below is a version of Whipple's $Q\leftrightarrow P$
transformation formula, but the proof given below is simpler than the
original~\cite{Whipple16}: no~integral representations are used.  Since the
algebraic change of variable from $L$ to~$R$ was introduced by Whipple, the
underlying curve is denoted~$\mathcal{W}_2$.

\begin{definition*}
  The algebraic $L$--$R$ curve $\mathcal{W}_2$ is the curve
  \begin{equation}
    \label{eq:w2}
    (L^2-1)(R^2-1)-1 = 0,
  \end{equation}
  which is rationally parametrized by
  \begin{displaymath}
    L =\frac{(p+1)^2 + (p-1)^2}{(p+1)^2 - (p-1)^2} = \frac{p^2+1}{2p}, \qquad
    R = \frac{p^2 + 1}{p^2-1},
  \end{displaymath}
  and is invariant under $L\leftrightarrow R$, $L\mapsto -L$ and $R\mapsto
  -R$, which are performed by $p\mapsto\allowbreak
  (p+\nobreak1)/\allowbreak(p-\nobreak1)$, $1/p$ and~$-p$, and under the
  group they generate (which can be viewed as a dihedral group of
  order~$8$, acting on $p=-1,0,1,\infty$).  An associated prefactor
  function $A=A(p)$, equal to unity when $p=1+\sqrt2$ and $L=R$, is
  \begin{displaymath}
    A(p) = (L^2-1)^{-1/4} = (R^2-1)^{1/4} = \sqrt{\frac{2p}{p^2-1}}.
  \end{displaymath}
\end{definition*}

\begin{theorem}
  \label{thm:w2}
  For each pair $u,v$ of Legendre functions listed below, an
  identity
  \begin{displaymath}
     u_{\alpha-1/2}^{-\beta}(L(p)) = \frac{\sqrt\pi}{\Gamma(\beta-\alpha+\frac12)}\,A(p)\, 
v_{\beta-1/2}^{-\alpha}(R(p))
  \end{displaymath}
  of type\/ $W_2$, coming from the curve\/ $\mathcal{W}_2$, holds for the
  specified range of values of the parameter\/ $p$.

\medskip\medskip
\begin{tabular}{llllll}
\hline
Label &$u_{\alpha-1/2}^{-\beta}$ &$v_{\beta-1/2}^{-\alpha}$ &$p$ range & $L$ range & $R$ range\\
\hline
\hline
(i) & $\sqrt{2}\,P_{\alpha-1/2}^{-\beta}$ & $(2/\pi)\widehat{Q}_{\beta-1/2}^{-\alpha}$ & $(1,\infty)$ & $1<L<\infty$ & $\infty>R>1$ \\
$(\overline{i})$ & $(2/\pi)\cos[(\beta-\alpha)\pi]\,\widehat Q_{\alpha-1/2}^{-\beta}$ & $\sqrt{2}\,P_{\beta-1/2}^{-\alpha}$ &  & & \\
\hline
\end{tabular}
\end{theorem}

\smallskip
To construct trigonometric versions of these identities, one substitutes
$L=\cosh\xi$ into the relation~(\ref{eq:w2}), and solves for~$R$ as a
function of~$\xi$; it equals $\coth\xi$.  This yields the following
versions, which for the sake of symmetry are expressed in~terms of Olver's
function~$\boldsymbol Q$ rather than~$\widehat Q$.  Owing to the
$L\leftrightarrow R$ invariance, they are equivalent to each other up~to
parametrization, and are also equivalent to Whipple's $Q\leftrightarrow P$
formula, eq.~(\ref{eq:whipple0}).

\begin{corollary}
  The following identities coming from\/ $\mathcal{W}_2$ hold for
  $\alpha,\beta\in\mathbb{C}$, when $\xi\in(0,\infty)$.  
  \begin{align*}
    &W_2(i):\quad 
    P_{\alpha-1/2}^{-\beta}(\cosh\xi) =
    \sqrt{2/\pi}\,\sqrt{{\mathrm{csch}}\,\xi}\:
    \boldsymbol{Q}_{\beta-1/2}^{-\alpha}\left(\coth\xi\right);&\\ &W_2(\overline{i}):\quad
    \boldsymbol{Q}_{\alpha-1/2}^{-\beta}(\cosh\xi) =
    \sqrt{\pi/2}\,\sqrt{{\mathrm{csch}}\,\xi}\:
    {P}_{\beta-1/2}^{-\alpha}\left(\coth\xi\right)
    .&
  \end{align*}
\end{corollary}

\begin{proof}[Proof of Theorem~\ref{thm:w2}]
Similarly to the proof of Theorem~\ref{thm:m}, this follows from lifting
Legendre's equation~(\ref{eq:ode}), now of degree $\nu=\alpha-1/2$ and
order $\mu=-\beta$, to the $p$-sphere, along the covering maps $L=L(p)$ and
$R=R(p)$.  (The latter lifting is followed by a linear change of the
dependent variable coming from the prefactor $A(p)=\sqrt{2p/(p^2-1)}$.)
The inverse image of the set of singular points $\{1,-1,\infty\}$ under
either $L$ or~$R$ is the subset $\{1,-1,\infty,0\}$ of the $p$-sphere,
which comprises the vertices of a square, and the left and right P-symbols
both turn~out to be
\begin{equation}
  \left\{\!
  \begin{array}{cccc|c}
    1 & -1 & \infty & 0 & p \\
    \hline
    -\beta & -\beta & -\alpha+1/2 & -\alpha+1/2 & \\
    \beta & \beta & \alpha+1/2 & \alpha+1/2 &
  \end{array}
  \!\right\},
\end{equation}
when account is taken of the fact that
$L^{-1}(1),L^{-1}(-1),R^{-1}(1),R^{-1}(-1)$, which respectively equal
$1,-1,\infty,0$, have double multiplicity.  However, since
$\{1,-1,\infty,0\}$ has cardinality greater than three, equality of the
P-symbols, though necessary, is not sufficient for the lifted equations
$\mathcal{E}_L$ and~$\mathcal{E}_R$ to equal each other.  They take the
identical form $\widetilde{\mathcal{L}}\tilde u=0$, in~fact they are both
\begin{equation}
  \frac{{\rm d}^2\tilde u}{{\rm d}p^2}
  +\frac{2p}{p^2-1}   \frac{{\rm d}\tilde u}{{\rm d}p}
  +\left[
    \frac{1-4\alpha^2}{4p^2} - \frac{4\beta^2}{(p^2-1)^2}
    \right]\tilde u
  = 0,
\end{equation}
but verifying this requires a separate calculation.

The identities of the theorem, each based on a pair $(u,v)$, come from the
table
\begin{equation}
\label{eq:tablefour}
\begin{tabular}{c||cccccc}
$p$ & $-\infty$ & $-1$ & $0$ & $1$ & $+\infty$ \\
\hline
$L(p)$ & $-\infty$ & $-1^*$ & $-\infty$/$+\infty$ & $1^*$ & $+\infty$ \\
$R(p)$ &$1^*$  & $+\infty$/$-\infty$ & $-1^*$ & $-\infty$/$+\infty$ & $1^*$,
\end{tabular}
\end{equation}
which shows how each of the four intervals $(-\infty,-1)$, $(-1,0)$,
$(0,1)$, $(1,+\infty)$, into which the real $p$-line is divided by the
singular points, is mapped monotonically onto a real $L$-interval and a
real $R$-interval.  (An asterisk indicates a change of direction.)  For
each $p$-interval, the possible $(u,v)$ are determined thus: if each of $u$
and~$v$ is to be one of $P,{\rm P},\xbar{\rm P},\widehat Q$, then the
defining singular points of $u(L)$ and~$v(R)$ (namely, $1$~for $P$
and~${\rm P}$, $-1$~for~$\xbar{\rm P}$ and $\infty$ for~$\widehat Q$) must
correspond, in the sense that both must be at the same end of the interval.
If so, $u(L(p))$ and $A(p)v(R(p))$ will be the same Frobenius solution in
the (two-dimensional) solution space of $\mathcal{E}_L=\mathcal{E}_R$, up
to a constant factor.

It should be noted that on the real axis, $P,\widehat Q$ are only defined
on~$(1,\infty)$, and ${\rm P},\xbar{\rm P}$ on~$(-1,1)$.  But for each
$p$-interval in~(\ref{eq:tablefour}) other than $(1,+\infty)$, either $L$
or~$R$ ranges between $-1$ and~$-\infty$.  So the only interval that will
work is $(1,+\infty)$; and for~it, there are exactly two possibilities for
$(u,v)$, namely $(P,\widehat Q)$ and~$(\widehat Q,P)$.  These yield
$W_2(i)$ and~$W_2(\overline{i})$.  For each identity, the prefactors in the
theorem are calculated by requiring agreement between the leading-order
behaviors of the left and right sides at the relevant singular point (i.e.,
at $p=1$ for~$W_2(i)$ and $p=\infty$ for~$W_2(\overline{i})$).
\end{proof}

\subsection{Whipple-like relations}

Theorem~\ref{thm:w4} below contains an unexpected pair of Whipple-like
identities, which are based on an algebraic curve of higher degree.  Since
its parameterization resembles that of the Whipple curve~$\mathcal{W}_2$,
with squares replaced by fourth powers, it is denoted~$\mathcal{W}_4$, and
the resulting identities are said to be of type~$W_4$.

\begin{definition*}
  The algebraic $L$--$R$ curve $\mathcal{W}_4$ is the curve
  \begin{equation}
    \label{eq:w4}
    16(L^2-1)(R^2-1)(4L^2 + 4R^2 - 5) - 1 = 0,
  \end{equation}
  which is rationally parametrized by
  \begin{displaymath}
    L =\frac{(p+1)^4 + (p-1)^4}{(p+1)^4 - (p-1)^4} = \frac{p^4 + 6p^2 + 1}{4p(p^2+1)}, \qquad
    R = \frac{p^4 + 1}{p^4-1},
  \end{displaymath}
  and is invariant under $L\leftrightarrow R$, $L\mapsto -L$ and $R\mapsto
  -R$, which are performed by $p\mapsto\allowbreak
  (p+\nobreak1)/\allowbreak(p-\nobreak1)$, $1/p$ and~$-p$, and under the
  group they generate (the same as for~$\mathcal{W}_2$).  An associated
  prefactor function $A=A(p)$, equal to unity when $p=1+\sqrt2$ and $L=R$,
  is
  \begin{displaymath}
    A(p) = \left[\frac{R^2-1}{L^2-1}\right]^{1/12} = \sqrt{\frac{2p}{p^2-1}}.
  \end{displaymath}
\end{definition*}

\begin{theorem}
  \label{thm:w4}
  For each pair $u,v$ of Legendre functions listed below, an
  identity
  \begin{displaymath}
     u_{2\alpha-1/2}^{-\alpha}(L(p)) = A(p)\, 
v_{2\alpha-1/2}^{-\alpha}(R(p))
  \end{displaymath}
  of type\/ $W_4$, coming from the curve\/ $\mathcal{W}_4$, holds for the
  specified range of values of the parameter\/ $p$.

\medskip\medskip
\begin{tabular}{llllll}
\hline
Label &$u_{2\alpha-1/2}^{-\alpha}$ &$v_{2\alpha-1/2}^{-\alpha}$ &$p$ range & $L$ range & $R$ range\\
\hline
\hline
(i) & ${2}P_{2\alpha-1/2}^{-\alpha}$ & $(2/\pi)\widehat{Q}_{2\alpha-1/2}^{-\alpha}$ & $(1,\infty)$ & $1<L<\infty$ & $\infty>R>1$ \\
$(\overline{i})$ & $(2/\pi)\widehat Q_{2\alpha-1/2}^{-\alpha}$ & ${2}P_{2\alpha-1/2}^{-\alpha}$ &  & & \\
\hline
\end{tabular}
\end{theorem}

\smallskip
To construct trigonometric versions of these identities, one substitutes
$L=\coth\xi$ into the relation~(\ref{eq:w4}), and solves for~$R$ as a
function of~$\xi$.  This yields the following, which owing to the
$L\leftrightarrow R$ invariance, are equivalent up~to parametrization.

\begin{corollary}
  The following identities coming from\/ $\mathcal{W}_4$ hold for
  $\alpha\in\mathbb{C}$, when $\xi\in(0,\infty)$.  
  \begin{align*}
    &W_4(i):\quad 
    2P_{2\alpha-1/2}^{-\alpha}(\coth\xi)
    =
    \sqrt{{\mathrm{sinh}}(\xi/2)}\:\,(2/\pi)
    \widehat{Q}_{2\alpha-1/2}^{-\alpha}\left(\frac{\cosh(\xi/2) + \mathrm{sech}(\xi/2)}2\right);&
\\
    &W_4(\overline{i}):\quad
    (2/\pi)\widehat{Q}_{2\alpha-1/2}^{-\alpha}(\coth\xi) 
    =
    \sqrt{{\mathrm{sinh}}(\xi/2)}\:\,2P_{2\alpha-1/2}^{-\alpha}\left(\frac{\cosh(\xi/2) + \mathrm{sech}(\xi/2)}2\right).&
  \end{align*}
\end{corollary}

\begin{proof}[Proof of Theorem~\ref{thm:w4}]
This closely resembles the proof of Theorem~\ref{thm:w2}.  The inverse
image of the set of singular points $\{1,-1,\infty\}$ under either $L$
or~$R$ is now the subset $\{1,-1,\infty,0,{\rm i},-{\rm i}\}$ of the
$p$-sphere, which comprises the vertices of a regular octahedron, and the
left and right P-symbols both turn~out to be
\begin{equation}
  \left\{\!
  \begin{array}{cccccc|c}
    1 & -1 & \infty & 0 & \mathrm{i} & -\mathrm{i} & p \\
    \hline
    -2\alpha & -2\alpha & -2\alpha+1/2 & -2\alpha+1/2 & -2\alpha+1/2 & -2\alpha+1/2 & \\
    2\alpha & 2\alpha & 2\alpha+1/2 & 2\alpha+1/2 &2\alpha+1/2 &2\alpha+1/2 &
  \end{array}
  \!\right\}.
\end{equation}
(It is an easy exercise to verify that if the order~$-\alpha$ were replaced
by~$-\beta$, as in Theorem~\ref{thm:w2}, the two P-symbols would differ;
which is why Theorem~\ref{thm:w4} has only one free parameter,
namely~$\alpha$.)  As before, the lifted equations
$\mathcal{E}_L,\mathcal{E}_R$ do not merely have the same P-symbol: they
are both
\begin{equation}
  \frac{{\rm d}^2\tilde u}{{\rm d}p^2}
  +\frac{2p}{p^2-1}   \frac{{\rm d}\tilde u}{{\rm d}p}
  +\left[
    \frac{(1-16\alpha^2)(p^2-1)^2}{4p^2(p^2+1)^2}
        -
        \frac{16\alpha^2}{(p^2-1)^2}
    \right]\tilde u
  = 0,
\end{equation}
by a separate calculation.  The remainder of the proof is similar; in fact
the table~(\ref{eq:tablefour}), showing how each of the four $p$-intervals
$(-\infty,-1)$, $(-1,0)$, $(0,1)$, $(1,+\infty)$ is mapped monotonically
onto a real $L$-interval and a real $R$-interval, is valid without change.
\end{proof}

\begin{remark*}
  One may wonder how the curve $\mathcal{W}_4$ was found, or equivalently
  the quartic covering maps $p\mapsto L(p),R(p)$.  In fact the identities
  of type~$W_4$ were found first, and the curve was engineered to provide a
  proof.  As the reader can verify, they follow from homographic identities
  (of type~$M$, above) by applying Whipple's transformation formula to both
  sides.

  Many (associated) Legendre functions of half-odd-integer degree and
  order, including ${\rm P}_{1/2}^{-1/2},\allowbreak{\rm
    P}_{5/2}^{-3/2},\allowbreak{\rm P}_{9/2}^{-5/2}$, have been tabulated
  for use in quantum mechanics~\cite{Hunter99}.  Up~to phase factors
  (see~(\ref{eq:ferrersp})), these are the same as
  ${P}_{1/2}^{-1/2},{P}_{5/2}^{-3/2},{P}_{9/2}^{-5/2}$.  By applying the
  identity $W_4(\overline i)$, one can easily compute
  $\widehat{Q}_{1/2}^{-1/2},{\widehat
    Q}_{5/2}^{-3/2},\widehat{Q}_{9/2}^{-5/2}$, which have not previously
  been tabulated.
\end{remark*}

\section{Derivation of main results}
\label{sec:derivation}

The two dozen identities in Theorems \ref{thm:i4}--\ref{thm:i3p}, arising
from algebraic curves $\mathcal{C}_r,\mathcal{C}_r'$ with $r=3,4,6$, are
proved by the technique developed in~\S\,\ref{sec:suppl}.  The key fact is
that in the four identities of each theorem, which come from a single
curve, the left and right functions $u(L(p))$, $A(p)v(R(p))$ satisfy the
same second-order differential equation, as functions of the parametrizing
variable~$p$.  This equality (i.e., $\mathcal{E}_L=\mathcal{E}_R$) is
consistent with $\mathcal{E}_L,\mathcal{E}_R$ having the same P-symbol,
which can be checked `on~the back of an envelope'; but due to each of these
lifted equations having more than three singular points on the $p$-sphere,
for full rigor they must be worked~out explicitly, and compared.

Once $\mathcal{E}_L=\mathcal{E}_R$ has been verified, the covering maps
$p\mapsto L(p),R(p)$ determine the associated identities: in particular,
which of the Legendre functions $P,\widehat Q$ or Ferrers functions ${\rm
  P},\xbar{\rm P}$ can appear as~$u,v$.  The algorithm for determining the
possible~$u,v$ was illustrated in~\S\,\ref{sec:suppl}.  For each real
$p$-interval delimited by real singular points, one checks whether the
$L$-range or $R$-range is~$(-\infty,-1)$; if so, the $p$-interval is
rejected.  An $L$-range or $R$-range that is $(-1,1)$, or a subset of~it,
corresponds to a Ferrers function, and similarly, $(1,\infty)$ corresponds
to a Legendre function.  For an identity to exist, the left and right
functions must be the same Frobenius solution, which means that their
defining singular points ($1$~for $P$ and~${\rm P}$, $-1$~for~$\xbar{\rm
  P}$ and $\infty$ for~$\widehat Q$) must appear at the same end of the
$p$-interval.  Any constant of proportionality needed between the two sides
is calculated by considering their asymptotic behavior at this singular
point (see~\S\,\ref{sec:asymptotics}).

The preceding algorithm suffices to derive or verify all the identities
of~\S\,\ref{sec:mainresults}, except for $I_r(\overline{i})$, $r=3,4,6$.
Anomalously, these relate $\widetilde P,{\rm Q}$, and a sketch of how they
are derived is deferred.  The data below on each curve (the singular points
of $\mathcal{E}_L=\mathcal{E}_R$, the characteristic exponent data, and the
table of monotone maps from $p$-intervals to $L,R$-intervals) should
suffice for the interested reader to confirm all identities other than
these.  It is exponent \emph{differences} that are supplied below, since
unlike exponent pairs they are unaffected by the replacement of
$v(R(p))$ by $A(p)v(R(p))$.

\smallskip
\emph{Signature-$4$ identities.}---On the curve~$\mathcal{C}_4$ viewed as
the $p$-sphere, the equation $\mathcal{E}_L=\mathcal{E}_R$ has singular
points $p=-1,0,1$.  The respective exponent differences are
$\alpha,2\alpha,\alpha$.  It also has a `removable' singular point at
$p=\infty$, at which the exponent difference is unity, but no Frobenius
solution behaves logarithmically.  The equation is
\begin{equation}
\frac{{\rm d}^2\tilde u}{{\rm d}p^2}
+
\left(\frac1{p+1} + \frac1p + \frac1{p-1}
\right)
\frac{{\rm d}\tilde u}{{\rm d}p}
+
\left[
\frac3{4(p^2-1)}
-\frac{\alpha^2}{p^2(p^2-1)^2}
\right]\tilde u
=0,
\end{equation}
by direct computation.  The real $p$-intervals and monotonic $p\mapsto L,R$
maps are tabulated as
\begin{equation}
\begin{tabular}{c||cccccc}
$p$ & $-\infty$ & $-1$ & $0$ & $1$ & $+\infty$ \\
\hline
$L(p)$ & $+\infty^*$ & $1$ & $-1^*$ & $1$ & $+\infty^*$ \\
$R(p)$ &$0$  & $-1$ & $-\infty$/$+\infty$ & $1$ & $0$.
\end{tabular}
\end{equation}
The $p$-interval $(1,+\infty)$ yields an identity: because $1<p<+\infty$
corresponds to $1<L<+\infty$ and ${1>R>0}$, its left and right functions
are $P,{\rm P}$.  This is identity $I_4(i)$ of Theorem~\ref{thm:i4}.  The
defining singular points of~$P,{\rm P}$ (respectively $L=1$, $R=1$) are
at~$p=1$, i.e., are at the same end, and the prefactor~$2^\alpha$ in the
theorem comes from requiring the two sides to agree at~$p=1$.  In the same
way, the $p$-interval $(0,1)$ yields both $I_4(ii)$ and
$I_4(\overline{ii})$, which respectively relate ${\rm P},P$ and~$\xbar{\rm
  P},\widehat Q$.  Their defining points are at $p=1$ and~$p=0$.  The
$p$-intervals $(-\infty,-1)$ and $(-1,0)$ also yield identities, but they
are related to the ones just found by $R\mapsto-R$, which is performed
by~$p\mapsto-p$.

On the curve~$\mathcal{C}_4'$, the equation $\mathcal{E}_L=\mathcal{E}_R$
has singular points $p=0,1,\infty$.  The respective exponent differences
are $0,0,0$.  It also has a removable singular point at $p=-1$.  The
equation is
\begin{equation}
\frac{{\rm d}^2\tilde u}{{\rm d}p^2}
+
\left(-\,\frac1{p+1} + \frac1p + \frac1{p-1}
\right)
\frac{{\rm d}\tilde u}{{\rm d}p}
+
\frac34\left[
\frac1{(p+1)^2} + \frac1{(p+1)}  - \frac1p
\right]\tilde u
=0,
\end{equation}
by direct computation.  (It is readily verified that if the left and right
order parameters equalled~$-\alpha$, as in the $I_4$ identities, then
$\mathcal{E}_L=\mathcal{E}_R$ only if $\alpha=0$; which is why
Theorem~\ref{thm:i4p} includes no free~$\alpha$ parameter.)  The real
$p$-intervals and monotonic $p\mapsto L,R$ maps are tabulated as
\begin{equation}
\begin{tabular}{c||cccccc}
$p$ & $-\infty$ & $-1$ & $0$ & $1$ & $+\infty$ \\
\hline
$L(p)$ & $-1$ & $-\infty^*$ & $-1$ & $1^*$ & $-1$ \\
$R(p)$ &$-1$  & $-3$ & $-\infty$/$+\infty$ & $1$ & $-1$.
\end{tabular}
\end{equation}
The $p$-interval $(1,+\infty)$ yields identities $I_4'(i)$ and
$I_4'(\overline{i})$ of Theorem~\ref{thm:i4p}, which respectively relate
${\rm P},{\rm P}$ and $\xbar{\rm P},\xbar{\rm P}$, and the $p$-interval
$(0,1)$ yields both $I_4'(ii)$ and $I_4'(\overline{ii})$, relating ${\rm
  P},{P}$ and $\xbar{\rm P},\widehat Q$.

\smallskip
\emph{Signature-$6$ identities.}---On the curve~$\mathcal{C}_6$ viewed as
the $p$-sphere, the equation $\mathcal{E}_L=\mathcal{E}_R$ has singular
points $p=-3,-1,0,1,3,\infty$, which are the vertices of a regular hexagon.
The respective exponent differences are
$\alpha,\alpha,4\alpha,\alpha,\alpha,4\alpha$.  It also has `apparent'
singular points at $p=\pm\sqrt3$, at each of which the exponent difference
is a non-zero integer other than unity (namely,~$2$), but no Frobenius
solution behaves logarithmically.  The equation is
\begin{multline}
\frac{{\rm d}^2\tilde u}{{\rm d}p^2}
+
\left(
\frac1{p+3} + \frac1{p+1} + \frac1{p} + \frac1{p-1} + \frac1{p-3} -
\frac{4p}{p^2-3}
\right)
\frac{{\rm d}\tilde u}{{\rm d}p}
\\
{}-
\left[
\frac{60\,p^2}{(p^2-1)(p^2-9)(p^2-3)^2}
+
\frac{4\alpha^2\,(p^2-3)^4}{p^2(p^2-1)^2(p^2-9)^2}
\right]\tilde u
=0,
\end{multline}
by direct computation.
The real $p$-intervals and monotonic $p\mapsto L,R$
maps are tabulated as
\begin{equation}
\begin{tabular}{c||cccccccccc}
$p$ & $-\infty$ & $-3$ & $-\sqrt3$ & $-1$ & $0$ & $1$ & $\sqrt3$ & $3$ & $+\infty$ \\
\hline
$L(p)$ & $1^*$ & $-1$ & $-\infty$/$+\infty$ & $1$ & $-1^*$ & $1$ &
$+\infty$/$-\infty$ & $-1$ & $1^*$
\\
$R(p)$ & $-\infty$ & $-1$ & $-\sqrt3/2^*$ & $-1$ & $-\infty$/$+\infty$ &
$1$ & $\sqrt3/2^*$ & $1$ & $+\infty$.
\end{tabular}
\end{equation}
The $p$-interval $(1,\sqrt3)$ yields identity $I_6(i)$ of
Theorem~\ref{thm:i6}, relating $P,{\rm P}$, and the $p$-interval $(0,1)$
yields both $I_6(ii)$ and $I_6(\overline{ii})$, which respectively relate
${\rm P},{P}$ and $\xbar{\rm P},\widehat Q$.  The $p$-intervals
$(-\sqrt3,-1)$, $(3,+\infty)$ also yield identities, but they are related
to the ones just found by $L\mapsto -L$ and $R\mapsto-R$, which are
performed by $p\mapsto3/p$ and~$p\mapsto-p$.

On the curve~$\mathcal{C}_6'$, the equation $\mathcal{E}_L=\mathcal{E}_R$
has singular points $p=-3,-1,0,1,3,\infty$.  The respective exponent
differences are $2\alpha,2\alpha,2\alpha,2\alpha,2\alpha,2\alpha$.  It also
has apparent singular points at $p=\pm\sqrt3\,{\rm i}$, with exponent
difference~$2$.  The equation is
\begin{multline}
\frac{{\rm d}^2\tilde u}{{\rm d}p^2}
+
\left(
\frac1{p+3} + \frac1{p+1} + \frac1{p} + \frac1{p-1} + \frac1{p-3} -
\frac{4p}{p^2+3}
\right)
\frac{{\rm d}\tilde u}{{\rm d}p}
\\
{}-
\left[
\frac{15}{(p^2+3)^2}
+
\frac{\alpha^2\,(p^2+3)^4}{p^2(p^2-1)^2(p^2-9)^2}
\right]\tilde u
=0,
\end{multline}
by direct computation.
The real $p$-intervals and monotonic $p\mapsto L,R$
maps are tabulated as
\begin{equation}
\begin{tabular}{c||cccccccc}
$p$ & $-\infty$ & $-3$ & $-1$ & $0$ & $1$ & $3$ & $+\infty$ \\
\hline
$L(p)$ & $1^*$ & $-1^*$ & $1^*$ & $-1^*$ & $1^*$ & $-1^*$ & $1^*$
\\
$R(p)$ & $+\infty$ & $1$ & $-1$ & $-\infty$/$+\infty$ & $1$ & $-1$ & $-\infty$.
\end{tabular}
\end{equation}
The $p$-interval $(1,3)$ yields identities $I_6'(i)$ and
$I_6'(\overline{i})$ of Theorem~\ref{thm:i6p}, which respectively relate
${\rm P},{\rm P}$ and $\xbar{\rm P},\xbar{\rm P}$, and the $p$-interval
$(0,1)$ yields both $I_6'(ii)$ and $I_6'(\overline{ii})$, relating ${\rm
  P},{P}$ and $\xbar{\rm P},\widehat Q$.  The $p$-intervals $(-3,-1)$,
$(-1,0)$ also yield identities, but they are related to the ones just found
by $L\mapsto -L$ and $R\mapsto-R$, which are performed by $p\mapsto-3/p$
and~$p\mapsto-p$.

\smallskip
\emph{Signature-$3$ identities.}---On the curve~$\mathcal{C}_3$ viewed as
the $p$-sphere, the equation $\mathcal{E}_L=\mathcal{E}_R$ has singular
points $p=-3,-1,0,1,3,\infty$.  The respective exponent differences are
$0,0,0,0,0,0$.  It also has removable singular points at $p=\pm\sqrt3$.
The equation is
\begin{multline}
\frac{{\rm d}^2\tilde u}{{\rm d}p^2}
+
\left(
\frac1{p+3} + \frac1{p+1} + \frac1{p} + \frac1{p-1} + \frac1{p-3} -
\frac{4p}{p^2-3}
\right)
\frac{{\rm d}\tilde u}{{\rm d}p}
\\
{}-
\left[
\frac{96\,p^2}{(p^2-1)(p^2-9)(p^2-3)^2}
\right]\tilde u
=0,
\end{multline}
by direct computation.  (As with $\mathcal{C}_4'$ above, there is no
evident generalization to non-zero order~$-\alpha$.)  The real
$p$-intervals and monotonic $p\mapsto L,R$ maps are tabulated as
\begin{equation}
\begin{tabular}{c||cccccccccc}
$p$ & $-\infty$ & $-3$ & $-\sqrt3$ & $-1$ & $0$ & $1$ & $\sqrt3$ & $3$ & $+\infty$ \\
\hline
$L(p)$ & $1^*$ & $-1$ & $-\infty$/$+\infty$ & $1$ & $-1^*$ & $1$ &
$+\infty$/$-\infty$ & $-1$ & $1^*$
\\
$R(p)$ & $+\infty$ & $1$ & $\sqrt3/2$ & $-1$ & $-\infty$/$+\infty$ &
$1$ & $-\sqrt3/2$ & $-1$ & $-\infty$.
\end{tabular}
\end{equation}
The $p$-interval $(1,\sqrt3)$ yields $I_3(i)$, relating $P,{\rm P}$, and
the $p$-interval $(0,1)$ yields both $I_3(ii)$ and $I_3(\overline{ii})$,
which respectively relate ${\rm P},{P}$ and $\xbar{\rm P},\widehat Q$.  The
$p$-intervals $(-\sqrt3,-1)$, $(-1,0)$ also yield identities, but they are
related to the ones just found by $R\mapsto-R$, which is performed
by~$p\mapsto-p$.

On the curve~$\mathcal{C}_3'$, the equation $\mathcal{E}_L=\mathcal{E}_R$
has singular points $p=-3,-1,0,1,3,\infty$.  The respective exponent
differences are $0,0,0,0,0,0$.  It also has removable singular points at
$p=\pm\sqrt3\,{\rm i}$.  By direct computation, the equation is
\begin{equation}
\frac{{\rm d}^2\tilde u}{{\rm d}p^2}
+
\left(
\frac1{p+3} + \frac1{p+1} + \frac1{p} + \frac1{p-1} + \frac1{p-3} -
\frac{4p}{p^2+3}
\right)
\frac{{\rm d}\tilde u}{{\rm d}p}
{}-
\left[
\frac{24}{(p^2+3)^2}
\right]\tilde u
=0.
\end{equation}
(As with $\mathcal{C}_3$, there is no evident generalization to non-zero
order~$-\alpha$.)  The table of real $p$-intervals and monotone $p\mapsto
L,R$ maps is the same as for~$\mathcal{C}_6'$, and the derivation of
identities is similar.

\smallskip
\emph{Finer asymptotics.}---It has now been explained how each identity in
Theorems \ref{thm:i4}--\ref{thm:i3p} is derived, except for
$I_r(\overline{i})$, $r=3,4,6$.  Each of these relates a $\widetilde P,{\rm
  Q}$, i.e., relates an \emph{ad~hoc} Legendre function on the left
(a~linear combination of~$P,\widehat Q$) to a Ferrers function of the
second kind, on the right.  Any identity $I_r(\overline{i})$ is anomalous
because, as the tables in Theorems \ref{thm:i4},~\ref{thm:i6}
and~\ref{thm:i3} show, its $R$-interval, over which the Ferrers argument
ranges, does not extend the entire way from $R=1$ to~$R=-1$.  This is why
the above proof technique, applied to this $R$-range and the corresponding
$p$-interval, produced only one identity (i.e., $I_r(i)$), which came by
requiring identical left and right asymptotics at the $R=1$ end: at the
singular point $p=1$.  The local behavior at the other end, which is not a
singular point, is not given by any simple formula.

This difficulty can be worked around by focusing on the $p=1$ end of the
relevant $p$-interval (which is $(1,\infty)$, $(1,3)$, $(1,\sqrt3)$ for
$r=4,6,3$), but employing finer asymptotic approximations.  The leading
behaviors of $P_\nu^{-\mu}(z),\allowbreak{\rm P}_\nu^{-\mu}(z)$ as~$z\to1$
are given in (\ref{eq:asympP}),(\ref{eq:asympPFerrers}).  Those of
$\widehat Q_\nu^{\mu}(z),\allowbreak{\rm Q}_\nu^{\mu}(z)$ as~$z\to1$ are
more difficult to compute.  (The point $z=1$ is not the defining singular
point for~$\widehat Q_\nu^\mu$, and the Ferrers function~${\rm Q}_\nu^\mu$
is not a Frobenius solution at any singular point.)  But one can exploit
the representation of~$\widehat Q_\nu^\mu$ as a combination
of~$P_\nu^\mu,P_\nu^{-\mu}$ \cite[3.3(10)]{Erdelyi53}, and that of~${\rm
  Q}_\nu^\mu$ as a combination of~${\rm P}_\nu^\mu,{\rm P}_\nu^{-\mu}$
\cite[3.4(14)]{Erdelyi53}.  One finds that if $\mu$~is not an integer and
$\nu\pm\mu$ are not negative integers,
\begin{equation}
\label{eq:512}
  \frac{(2/\pi)\sin(\mu\pi)}{\Gamma(\nu+\mu+1)}\widehat Q_\nu^\mu(z)\sim
  \frac{\left[(z-1)/2\right]^{-\mu/2}}{\Gamma(1-\mu)\Gamma(\nu+\mu+1)}
  -
  \frac{\left[(z-1)/2\right]^{\mu/2}}{\Gamma(1+\mu)\Gamma(\nu-\mu+1)}
  ,\qquad z\to1,
\end{equation}
and a similar statement holds with $\widehat Q_\nu^\mu$ and $z-1$ replaced
by ${\rm Q}_\nu^\mu$ and~$1-z$, if the first term on the right is
multiplied by~$\cos(\mu\pi)$.  Such asymptotic statements must be
interpreted with care: the two terms are the leading terms of distinct
Frobenius series, from exponents $-\mu/2,+\mu/2$.

It is easily checked that if in Theorems \ref{thm:i4},~\ref{thm:i6}
and~\ref{thm:i3}, the right function~$v$ equals $2/\pi$ times the specified
Ferrers function~${\rm Q}$, and the left function~$u$ equals $\csc(\pi/r)$
times the specified function~$\widetilde P$, the two sides of the identity
$I_r(\overline{i})$ will have the same \emph{fine} asymptotics at~$p=1$:
the coefficients of each of the two Frobenius solutions will be in
agreement.  In fact, it was to obtain this agreement that the \emph{ad~hoc}
Legendre function~$\widetilde P$ was defined in~(\ref{eq:widetildeP}) as it
was, as a certain combination of~$P,\widehat Q$.

In deriving identity $I_3(\overline{i})$ of Theorem~\ref{thm:i3}, a
modified approach is needed.  This identity relates $\widetilde P_{-1/3}$
to~${\rm Q}_{-1/2}$, with both functions of order zero (there is no
$\alpha$~parameter).  In the asymptotic development of $\widehat
Q_{-1/3}(z),\allowbreak {\rm Q}_{-1/2}(z)$ as~$z\to1$, the Frobenius
solutions $(z-\nobreak1)^{-\mu/2},\allowbreak (z-\nobreak1)^{\mu/2}$
of~(\ref{eq:512}) are replaced by $1,\ln(z-\nobreak1)$; see
\cite[\S\,3.9.2]{Erdelyi53}.  The modifications are straightforward.

\section{Elliptic integral representations}
\label{sec:ellreps}

The now-proved identities of section~\ref{sec:mainresults}, joined with
differential recurrences for Legendre and Ferrers functions, lead to useful
representations in terms of the first and second complete elliptic
integrals, $K=K({\rm m})$ and $E=E({\rm m})$, the argument ${\rm
  m}$~denoting the elliptic modular parameter.

\begin{theorem}
  The Legendre functions\/ $P_\nu^m(\cosh\xi),\widehat Q_\nu^m(\cosh\xi)$
  and Ferrers functions\/ ${\rm P}_\nu^m(\cos\theta),\allowbreak {\rm
    Q}_\nu^m(\cos\theta)$, where the degree\/ $\nu$ differs by\/ $\pm1/r$
  {\rm(}$r=2,3,4,6${\rm)} from an integer and the order\/ $m$ is an
  integer, can be expressed in closed form in terms of the complete
  elliptic integrals\/ $K,E$.
\end{theorem}
\begin{proof}
  The case $r=2$ is well-known (the Legendre functions of half-odd-integer
  degree and integer order are the classical toroidal functions).  The
  fundamental representations are
\begin{subequations}
  \begin{alignat}{2}
  \label{eq:fundrepsa}
    P_{-1/2}(\cosh\xi)&=(2/\pi)\,{\rm sech}(\xi/2)K(\tanh^2(\xi/2)), \qquad& \widehat Q_{-1/2}(\cosh\xi)&= 2\,{\rm e}^{-\xi/2}K({\rm e}^{-2\xi}),
  \\
    {\rm P}_{-1/2}(\cos\theta)&=(2/\pi)K(\sin^2(\theta/2)), \qquad & {\rm Q}_{-1/2}(\cos\theta)&= K(\cos^2(\theta/2)),
  \label{eq:fundrepsb}
  \end{alignat}
\end{subequations}
  and more general $\nu,m$ are handled by applying standard differential
  recurrences on the degree and order.  Let $F_\nu^\mu$ denote
  $P_\nu^m(\cosh\xi)$ or~$\widehat Q_\nu^m(\cosh\xi)$, and ${\rm
    F}_\nu^\mu$ denote ${\rm P}_\nu^m(\cos\theta)$ or~${\rm
    Q}_\nu^m(\cos\theta)$.  
  The order recurrences are
  \begin{subequations}
    \begin{align}
      M^\pm F_\nu^\mu &= s\, C^\pm\, F_\nu^{\mu\pm1},\\
      {\rm M}^\pm {\rm F}_\nu^\mu &= \pm C^\pm\, F_\nu^{\mu\pm1},
    \end{align}
  \end{subequations}
  where the Legendre and Ferrers `ladder' operators for the order,
  $M^{\pm}$ and~${\rm M}^{\pm}$, are defined (with $D_\xi = {\rm d}/{\rm
    d}\xi$ and $D_\theta = {\rm d}/{\rm d}\theta$) by
  \begin{subequations}
    \begin{align}
      M^\pm &= D_\xi \mp \mu \coth\xi,\\ {\rm M}^\pm &= D_\theta \mp \mu
      \cot\theta.
    \end{align}
  \end{subequations}
  The constant of proportionality $C^-$ equals
  $(\nu+\frac12)^2-(\mu-\frac12)^2$, $C^+$ equals unity and the sign
  factor~$s$ has the following meaning: $s=1,-1$ for $F=P,\widehat Q$.  The
  degree recurrences are
  \begin{subequations}
    \begin{align}
      M_{\pm} F_\nu^\mu &= \left[\mp(\nu+\tfrac12) + (\mu-\tfrac12)\right]\,F_{\nu\pm1}^\mu,\\
      {\rm M}_{\pm} {\rm F}_\nu^\mu &= \left[\mp(\nu+\tfrac12) +
        (\mu-\tfrac12)\right]\,{\rm F}_{\nu\pm1}^\mu,
    \end{align}
  \end{subequations}
  where the ladder operators for the degree, $M_{\pm}$ and~${\rm M}_{\pm}$,
  are given by
  \begin{subequations}
    \begin{align}
      M_\pm &=  -(\sinh\xi)D_\xi - \left[\tfrac12 \pm
        (\nu+\tfrac12)\right]\cosh\xi, \\
      {\rm M}_\pm &=  -(\sin\theta)D_\theta - \left[\tfrac12 \pm
        (\nu+\tfrac12)\right]\cos\theta.
    \end{align}
  \end{subequations}
  By applying these recurrences to any of $P_\nu^m(\cosh\xi),\widehat
  Q_\nu^m(\cosh\xi)$ or ${\rm P}_\nu^m(\cos\theta),{\rm
    Q}_\nu^m(\cos\theta)$, where $\nu$~is a half-odd-integer, one can
  express it in~terms of the corresponding $P_{-1/2}(\cosh\xi),\widehat
  Q_{-1/2}(\cosh\xi)$ or ${\rm P}_{-1/2}(\cos\theta),{\rm
    Q}_{-1/2}(\cos\theta)$, and its derivatives.  By then exploiting
  formulas (\ref{eq:fundrepsa}), (\ref{eq:fundrepsb}) and the known
  differentiation formulas for $K=K({\rm m})$ and $E=E({\rm m})$, which are
  \begin{equation}
    \frac{{\rm d}}{{\rm d}{\rm m}} K = \frac{E-{\rm m}K}{2\,{\rm m}(1-{\rm m})},
    \qquad
    \frac{{\rm d}}{{\rm d}{\rm m}} E = \frac{E-K}{2\,{\rm m}},
  \end{equation}
  the desired representation is produced.

  The preceding algorithm is easily extended from $r=2$ (the classical
  case) to $r=3,4,6$.  Suppose one were given a Legendre function
  $P_\nu^m(\cosh\xi)$ or $\widehat Q_\nu^m(\cosh\xi)$ with
  $\nu\in\mathbb{Z}-1/r$ and~$m\in\mathbb{Z}$.  One would first apply the
  recurrences on the degree and order, to express it in~terms of
  $P_{-1/r}(\cosh\xi)$ or $\widehat Q_{-1/r}(\cosh\xi)$, and its
  derivatives.  The pair of signature-$r$ identities $I_r(i),\allowbreak
  I_r(\overline{i})$, which are found in Corollaries \ref{cor:i4},
  \ref{cor:i6},~\ref{cor:i3} for $r=4,6,3$ respectively, will express these
  Legendre functions in~terms of the Ferrers functions ${\rm P}_{-1/2},{\rm
    Q}_{-1/2}$.  (Recall that $\widetilde P$, which appears on the left
  of~$I_r(i)$, is a linear combination of~$P,\widehat Q$;
  see~(\ref{eq:widetildeP}).)  Thus the cases $r=3,4,6$ reduce to the
  classical case.

  If one were given a Ferrers function, one of ${\rm P}_\nu^m(\cos\theta)$
  or ${\rm Q}_\nu^m(\cos\theta)$ with $\nu\in\mathbb{Z}-1/r$ and
  ${m\in\mathbb{Z}}$, the reduction would be similar, but one of the other
  three pairs of signature-$r$ identities (say, the pair
  $I_r(ii),I_r(\overline{ii})$) would be used.  The identities in each of
  these pairs have ${\rm P},\xbar{\rm P}$ on their left sides; but since
  $\xbar{\rm P}$~is a combination of~${\rm P},{\rm Q}$
  (see~(\ref{eq:barP})), this is sufficient for reduction.

  The only thing that remains to be explained is how to handle the case
  when the degree~$\nu$ differs by~$+1/r$ rather than~$-1/r$ from an
  integer.  The additional effort required is minor.  The functions
  $P_\nu^\mu,{\rm P}_\nu^\mu$ are unaffected by the replacement of $\nu$
  by~$-\nu-1$, which interchanges the two cases; and for $\widehat
  Q_\nu^\mu,{\rm Q}_\nu^\mu$, applying the
  identities\cite[3.3(9)\ and\ 3.4(16)]{Erdelyi53}
  \begin{subequations}
  \begin{gather}
    \widehat Q_{-\nu-1}^\mu - \widehat Q_\nu^\mu = \cos(\nu\pi)\Gamma(\nu+\mu+1)\Gamma(\mu-\nu)\,P_\nu^{-\mu},\\
    \sin\left[(\nu-\mu)\pi\right]\,
    {\rm Q}_{-\nu-1}^\mu - 
    \sin\left[(\nu+\mu)\pi\right]\,
    {\rm Q}_\nu^\mu 
    =
    -\pi \cos(\nu\pi)\cos(\mu\pi)\,P_\nu^\mu
  \end{gather}
  \end{subequations}
  reduces either case to the other.
\end{proof}

The algorithm in this proof is not optimal when $r=4,6$.  For these two
values of~$r$, the Ferrers functions ${\rm P}_{-1/r}^m,\xbar{\rm
  P}_{-1/r}^m$ for any $m\in\mathbb{Z}$ can be reduced directly to the
toroidal functions $P_{-m-1/2}^m, \widehat Q_{-m-1/2}^m$,
resp.\ $P_{-2m-1/2}^m, \widehat Q_{-2m-1/2}^m$, by identities
$I_r(ii),I_r(\overline{ii})$, though there is no analogous reduction for
$r=3$ if $m$~is non-zero.  This enhancement for $r=4,6$ may be of numerical
relevance, since the recurrences for Legendre and Ferrers functions are
often numerically unstable, and modern schemes for evaluating toroidal
functions do not employ them~\cite{Gil2000}.

\section{Algebraic Legendre functions}
\label{sec:algebraic}
One of the identities of section~\ref{sec:mainresults}, the signature-$6$
identity $I_6'(\overline{ii})$ of Theorem~\ref{thm:i6p} and
Corollary~\ref{cor:i6p}, can be employed to generate closed-form
expressions for ${\rm P}_{-1/6}^{-1/4}$, ${P}_{-1/6}^{-1/4}$ and
$\widehat{Q}_{-1/4}^{-1/3}$.  These turn~out to be elementary
(specifically, algebraic) functions of their arguments, so the expressions
are conceptually as~well as practically simpler than the ones for
$P_\nu^m,\widehat Q_\nu^m,\allowbreak {\rm P}_\nu^m,\widehat {\rm Q}_\nu^m$
(with $\nu\in\mathbb{Z}\pm\nobreak1/2$) that were covered in the last
section.  No~elliptic integrals are involved.

The key fact is that while identity $I_6'(\overline{ii})$ transforms $\xbar
{\rm P}_{-1/6}^{-\alpha}$ to~$\widehat Q_{\alpha-1/2}^{-2\alpha}$, a
closed-form expression for~$\widehat Q_\nu^\mu$ (and also $P_\nu^\mu,{\rm
  P}_\nu^\mu$ and~${\rm Q}_\nu^\mu$) is available whenever the order~$\mu$
is a half-odd-integer.  This expression involves only elementary functions
\cite[\S\,14.5(iii)]{Olver2010}.  That any Legendre or Ferrers function
with (i)~$\mu\in\mathbb{Z}+1/2$, or (ii)~$\nu\in\mathbb{Z}$, can be
represented without quadratures is an important result~\cite{Poole30}.  In
this statement, cases (i) and~(ii) are related by Whipple's transformation.

\begin{theorem}
  The following formulas hold when\/ $\theta\in(0,\pi)$ and\/
  $\xi\in(0,\infty)$.
  \begin{align*}
    {\rm P}_{-1/6}^{-1/4}(\cos\theta) &= 3^{3/4}\Gamma(5/4)^{-1}\,(\sin\theta)^{-1/4}\left[\cos(\theta/3)-\sqrt{\frac{\sin\theta}{3\sin(\theta/3)}}\,\right]^{1/4},\\
    {P}_{-1/6}^{-1/4}(\cosh\xi) &= 3^{3/4}\Gamma(5/4)^{-1}\,(\sinh\xi)^{-1/4}\left[-\cosh(\xi/3)+\sqrt{\frac{\sinh\xi}{3\sinh(\xi/3)}}\,\right]^{1/4}.
  \end{align*}
  Moreover,
  \begin{displaymath}
    \widehat Q_{-1/4}^{-1/3}(\coth\xi) = 
C\,(\sinh\xi)^{-1/4}\left[-\cosh(\xi/3)+\sqrt{\frac{\sinh\xi}{3\sinh(\xi/3)}}\,\right]^{1/4},
  \end{displaymath}
  where\/ $C=3^{3/4}\sqrt{\pi/2}\,\Gamma(5/12)/\Gamma(5/4) =
  2^{3/4}3^{9/8}\Gamma(2/3)\sqrt{\sqrt3 - 1}$ is the constant prefactor.
\label{prop:71}
\end{theorem}

\begin{remark*}
  To obtain explicit formulas when the degree and order differ by integers
  from those shown in this theorem, one would apply differential
  recurrences, as in the last section.
\end{remark*}

\begin{proof}
  An explicit formula for $\widehat Q_{-1/4}^{-1/2}(\cosh\xi)$ follows from
  \cite[eq.~14.5.17]{Olver2010}.  In algebraic rather than trigonometric
    form, it is
    \begin{equation}
      \widehat Q_{-1/4}^{-1/2}(z)= {\rm i}\,Q_{-1/4}^{-1/2}(z)= 4\sqrt{{\pi}/2} \left[ (z^2-1)^{-1}
        (z-\sqrt{z^2-1}) \right]^{1/4}, \qquad z>1.
    \end{equation}
    Substituting this into the right side of the $\alpha=1/4$ case
    of~$I_6'(\overline{ii})$, and performing some lengthy trigonometric
    manipulations, yields a formula for $\xbar{\rm
      P}_{-1/6}^{-1/4}(\cos\theta) = {\rm P}_{-1/6}^{-1/4}(-\cos\theta)$,
    which upon $\theta$~being replaced by $\theta+\pi$, becomes the one for
    ${\rm P}_{-1/6}^{-1/4}(\cos\theta)$ given in the theorem.

    The formula for ${P}_{-1/6}^{-1/4}(\cosh\xi)$ comes by analytic
    continuation (informally, by setting $\theta={\rm i}\,\xi$).  The one
    for $\widehat Q_{-1/4}^{-1/3}(\coth\xi)$, with the first value given
    for the prefactor~$C$, then comes by applying Whipple's transformation.
    The equality of the two values given for~$C$ comes from a
    gamma-function identity~\cite[p.~270]{Vidunas2005a}.
\end{proof}

The formulas of Theorem~\ref{prop:71} can be written in algebraic form,
since $\cosh(\xi/3),\sinh(\xi/3)$ are algebraic functions of
$\cosh\xi$,~etc.  The significance of ${\rm P}_{-1/6}^{-1/4}(z)$,
${P}_{-1/6}^{-1/4}(z)$ and~$\widehat{Q}_{-1/4}^{-1/3}(z)$ being elementary
functions of~$z$, expressible in~terms of radicals, is the following.
Legendre's differential equation~(\ref{eq:ode}) on the Riemann
sphere~$\mathbb{P}^1$ is of the `hypergeometric' sort, with only three
singular points, $z=\pm1$ and~$z=\infty$; and their respective
characteristic exponent differences are~$\mu,\mu,2\nu+1$.  It is a
classical result of Schwarz (see \cite[\S\,2.7.2]{Erdelyi53},
\cite[Chap.~VII]{Poole36} and~\cite{Kimura69}) that for a differential
equation of the hypergeometric sort to have \emph{only algebraic
  solutions}, its unordered triple of exponent differences must be one of
$15$~types, traditionally numbered I--XV\null.  The case when
$(\nu,\mu)=(-1/4,-1/3)$ and $(\mu,\mu,2\nu+1)=(-1/3,-1/3,1/2)$ is of
Schwarz's type~II, and the case when $(\nu,\mu)=(-1/6,-1/4)$ and
$(\mu,\mu,2\nu+1)=(-1/4,-1/4,2/3)$ is of type~V\null.

For each type in Schwarz's list, there is a (projective) monodromy group:
the group of permutations of the branches of an algebraic solution that is
generated by loops around the three singular points.  (Strictly speaking,
the algebraic function here is not a solution of the equation, but the
ratio of any independent pair of solutions; which is the import of the term
`projective.')  For Schwarz's types II and~V, the respective groups are
tetrahedral and octahedral: they are isomorphic to the symmetry groups of
the tetrahedron and octahedron, which are of orders $12$ and~$24$.  It is
no~accident that as an algebraic function of $\cos\theta$ or $\cosh\xi$,
each right side in Theorem~\ref{prop:71} has a multiple of $12$ branches.

An interesting consequence of the formula for $P_{-1/6}^{-1/4}$ is a
formula in terms of radicals for an octahedral case of the Gauss
hypergeometric function,~${}_2F_1$.  Taking into account the relation
\begin{equation}
P_\nu^\mu(z) = \frac1{\Gamma(1-\mu)}\,\left(
\frac{z+1}{z-1}\right)^{\mu/2} {}_2F_1\left(-\nu,\nu+1;\,1-\mu;\,\frac{1-z}2\right),
\end{equation}
and using Cardano's formula to solve for $\cosh(\xi/3),\sinh(\xi/3)$ in
terms of $z=\cosh\xi$, one deduces
\begin{equation}
\begin{gathered}
{}_2F_1(\tfrac16,\tfrac56;\,\tfrac54;\,x) =
3^{3/4}(-2x)^{-1/4}
\left[
-\,\frac{A^{1/3}+A^{-1/3}}2 + \sqrt{ \frac{1+A^{2/3}+A^{-2/3}}3 }
\,\right]^{1/4},
\\
A = \bigl( \sqrt{-2x} + \sqrt{-2(x-1)}\, \bigr)^2 \!\bigm/ 2.
\end{gathered}
\label{eq:octahedral}
\end{equation}
This holds when $x<0$, and in fact on the complex $x$-plane with cut
$[1,\infty)$, on which the left side is analytic in~$x$; provided, that~is,
  that the branch of each radical is appropriately chosen.

It has long been known how to obtain \emph{parametric} expressions for
algebraic hypergeometric functions \cite[Chap.~VII]{Poole36}, and a
parametric formula for ${}_2F_1(\tfrac16,\tfrac56;\,\tfrac54;\,x)$ has
recently been derived \cite[eq.~(2.8)]{Vidunas2013}.  But the explicit
formula~(\ref{eq:octahedral}) may be more useful.  It does not appear in
the best-known data base of closed-form expressions for hypergeometric
functions \cite{Prudnikov86c}, or in the data base generated by
Roach~\cite{Roach96}, which is currently available
at~\texttt{www.planetquantum.com}.

\section{A curiosity}
\label{sec:curiosity}

Until this point, each Legendre transformation formula derived in this
paper has been related at~least loosely to the function transformations in
Ramanujan's theory of signature-$r$ elliptic integrals.  More exotic
Legendre transformations exist, as the curious theorem and corollary below
reveal.  They relate ${\rm P}_{-1/4}^{-1/10}$ to ${\rm P}_{-1/4}^{-1/5}$
(or~${P}_{-1/4}^{-1/5}$), despite neither of these functions being an
algebraic function of its argument, expressible in~terms of complete
elliptic integrals, or indeed (by~results of Kimura~\cite{Kimura69})
expressible at~all in~terms of elementary functions and their integrals.

\begin{definition*}
  The algebraic $L$--$R$ curve $\mathcal{X}$ is defined by the rational
  parametrization
  \begin{equation}
    \label{eq:lastcovering}
    \begin{aligned}
      L &= \frac{1-p^2}{1+p^2}=1-\frac{2p^2}{1+p^2} = -1 + \frac2{1+p^2},\\
      R &= 1-\frac{2p(2+p)^5}{(1+p^2)(1+11p-p^2)^2}
      = -1+\frac{2(1-2p)^5}{(1+p^2)(1+11p-p^2)^2},
    \end{aligned}
  \end{equation}
  and is invariant under $(L,R)\mapsto (-L,-R)$, which is performed by
  $p\mapsto-1/p$.  An associated prefactor function $A=A(p)$, equal to
  unity when $p=0$ and $(L,R)=(1,1)$, is
  \begin{displaymath}
    A(p) = \sqrt{\frac{(2+p)(1-2p)}{2(1+11p-p^2)}}.
  \end{displaymath}
\end{definition*}

\begin{theorem}
  \label{thm:x}
  For each pair $u,v$ of Legendre or Ferrers functions listed below, an
  identity
  \begin{displaymath}
    u_{-1/4}^{-1/10}(L(p)) = \frac{\Gamma(6/5)}{\sqrt2\,\Gamma(11/10)}\, A(p)\, v_{-1/4}^{-1/5}(R(p))
  \end{displaymath}
  of type\/ $X$, coming from the curve\/ $\mathcal{X}$, holds for the
  specified range of values of the parameter\/ $p$.

\medskip\medskip
\begin{tabular}{llllll}
\hline
Label &$u_{-1/4}^{-1/10}$ &$v_{-1/4}^{-1/5}$ &$p$ range & $L$ range & $R$ range\\
\hline
\hline
(i) & ${\rm P}_{-1/4}^{-1/10}$ & ${\rm P}_{-1/4}^{-1/5}$ & $(0,\frac12)$ & $1>L>\frac35$ & $1>R>-1$ \\
\hline
(ii) & ${\rm P}_{-1/4}^{-1/10}$ & ${P}_{-1/4}^{-1/5}$ & $(-\frac12(5\sqrt5-11),0)$ & $11/(5\sqrt5)<L<1$ & $\infty>R>1$ \\
\hline
\end{tabular}
\end{theorem}

\smallskip
To construct trigonometric versions of these identities, one substitutes
$L=\cos\theta$ into the parametrization, obtaining $p=\pm\tan(\theta/2)$,
which can be used for $X(i)$ and $X(ii)$ respectively, and then writes $R$
in terms of~$\theta$.  This yields the following.

\begin{corollary}
  The following identities coming from\/ $\mathcal{X}$ hold when\/
  $\theta\in(0,\tan^{-1}(4/3))$ and\/ $\theta\in(0,\tan^{-1}(2/11))$,
  respectively.
  \begin{align*}
    &X(i):\quad 
    {\rm P}_{-1/4}^{-1/10}(\cos\theta)
    = C\sqrt{\frac{4\cos\theta - 3\sin\theta}{2\cos\theta + 11\sin\theta}}\,
    {\rm P}_{-1/4}^{-1/5}\left(
    1 - 8\,\frac{[2\cos(\theta/2) + \sin(\theta/2)]^5}{(2\cos\theta+11\sin\theta)^2}\sin(\theta/2)
    \right)
;&
\\
    &X(ii):\quad
    {\rm P}_{-1/4}^{-1/10}(\cos\theta)
    = C\sqrt{\frac{4\cos\theta + 3\sin\theta}{2\cos\theta - 11\sin\theta}}\,
    {P}_{-1/4}^{-1/5}\left(
    1 + 8\,\frac{[2\cos(\theta/2) - \sin(\theta/2)]^5}{(2\cos\theta-11\sin\theta)^2}\sin(\theta/2)
    \right)
.&
  \end{align*}
  In both, the constant prefactor\/ $C$ equals
  $\Gamma(6/5)/\,{2}\Gamma(11/10)$.
\end{corollary}

\begin{proof}[Proof of Theorem~\ref{thm:x}]
This resembles the proofs in~\S\,\ref{sec:derivation} of the main results,
and will only be sketched.  On the curve $\mathcal{X}$ viewed as the
$p$-sphere, the equation $\mathcal{E}_L=\mathcal{E}_R$ has singular points
$p=0,\infty,{\rm i},-{\rm i}$, which are the vertices of a square.  The
respective exponent differences are $\frac15,\frac15,\frac12,\frac12$.  The
lifted equations $\mathcal{E}_L,\mathcal{E}_R$ for $\tilde u=\tilde
u(p)=u(L(p))$ and $A(p)v(R(p))$ both take the form
\begin{equation}
\label{eq:lastode}
\frac{{\rm d}^2\tilde u}{{\rm d}p^2}
+
\frac1p
\frac{{\rm d}\tilde u}{{\rm d}p}
-
\left[
\frac1{100p^2} + \frac{3}{4(p^2+1)^2}
\right]\tilde u
=0,
\end{equation}
by a separate computation.  The real $p$-intervals and monotonic $p\mapsto
L,R$ maps are tabulated as
\begin{equation}
\label{eq:lasttabular}
\begin{tabular}{c||ccccccc}
$p$ & $-\infty$ & $-2$ & $-\frac12(5\sqrt5-11)$ & $0$ & $\frac12$ & $\frac12(5\sqrt5+11)$ & $+\infty$ \\
\hline
$L(p)$ & $-1^*$ & $-\frac35$ & $11/(5\sqrt5)$ & $1^*$ & $\frac35$ & $-11/(5\sqrt5)$ &
$-1^*$
\\
$R(p)$ & $-1$ & $1$ & $+\infty^*$ & $1$ & $-1$ &
$-\infty^*$ & $-1$.
\end{tabular}
\end{equation}
The $p$-interval $(0,\frac12)$ yields the identity $X(i)$, relating ${\rm
  P},{\rm P}$, and the $p$-interval $(-\frac12(5\sqrt5-11),0)$ yields
$X(ii)$, relating ${\rm P},{P}$.  The $p$-interval $(-\infty,-2)$ also
yields an identity, but this third identity, relating $\xbar{\rm
  P},\xbar{\rm P}$, is related to~$X(i)$ by $(L,R)\mapsto-(L,R)$, which is
performed by~$p\mapsto-1/p$.

From the covering map $R=R(p)$ of~(\ref{eq:lastcovering}) and the
table~(\ref{eq:lasttabular}), one would expect that ${p=-2}$
(in~$R^{-1}(1)$), $p=\frac12$ (in~$R^{-1}(-1)$) and
$p=\mp\frac12(5\sqrt5\mp11)$ (in~$R^{-1}(\infty)$) would also be singular
points of $\mathcal{E}_L=\mathcal{E}_R$.  But they are ordinary points,
i.e.\ singular points that have `disappeared' upon lifting, as explained
in~\S\,\ref{sec:suppl}, on account of their characteristic exponents
being~$0,1$.

Since the cardinality of the set of singular points $\{0,\infty,{\rm
  i},-{\rm i}\}$ is greater than three, the equality of
$\mathcal{E}_L,\mathcal{E}_R$ cannot be verified by an (easy) comparison of
their P-symbols; that they are equal because both are of the
form~(\ref{eq:lastode}) must be checked explicitly.
\end{proof}

The rather mysterious algebraic curve~$\mathcal{X}$ was discovered
heuristically; although, as one conjectures from a close examination of the
resulting identities $X(i),X(ii)$, their existence is in some way tied to
the equations $3^2+4^2=5^2$ and $2^2+11^2=5^3$.  The existence of other
exotic algebraic curves that lead to Legendre or Ferrers identities will be
explored elsewhere.




\end{document}